\documentclass[12pt]{article}
\usepackage[english]{babel}
\usepackage{url}
\usepackage{amssymb,amsmath}
\usepackage{tikz}
\usepackage{tocloft}
\usepackage{mdwlist}
\newtheorem{theorem}{Theorem}[section]

\newtheorem{prop}[theorem]{Proposition}
\newtheorem{lemma}[theorem]{Lemma}
\newtheorem{cor}[theorem]{Corollary}

\newtheorem{rem}[theorem]{Remark}

\newenvironment{proof}{\par\noindent\textbf{Proof}\hspace{1em}}{\qed}
\def\<{\langle}
\def\>{\rangle}
\newcommand{\PAG}{\mathbb{P}}
\newcommand{\cC}{\mathcal{Cy}}

\newcommand{\cE}{\mathcal{E}}
\newcommand{\cF}{\mathcal{F}}

\newcommand{\A}{\mathbb{A}}
\newcommand{\K}{\mathbb{K}}
\newcommand{\C}{\mathbb{C}}

\newcommand{\cS}{\mathcal{S}}

\newcommand{\cH}{\mathcal{H}}
\newcommand{\cD}{\mathcal{D}}

\newcommand{\cG}{\mathcal{G}}
\newcommand{\cV}{\mathcal{V}}
\newcommand{\cL}{\mathcal{L}}

\renewcommand{\O}{\mathbb{O}}

\newcommand{\N}{\mathbb{N}}

\def\EQ{symp}
\def\qed{{\hfill\hphantom{.}\nobreak\hfill$\Box$}}

\begin{document}

\author{J.~Schillewaert\thanks{Research supported by Marie Curie IEF grant GELATI (EC grant nr 328178)} \and H.~Van Maldeghem\thanks{Corresponding author; Tel:+32 9 264 49 11; Fax:+32 9 264 49 93}}
\title{\bf On the varieties of the second row of the split Freudenthal-Tits Magic Square}
\date{\footnotesize\begin{tabular}{ll}Department of Mathematics, &Department of Mathematics,\\
Imperial College, & Ghent University,\\
South Kensington Campus,& Krijgslaan 281-S22,\\
 London SW7-2AZ,& B-9000 Ghent,\\
UK & BELGIUM\\
\texttt{jschillewaert@gmail.com} & 
\texttt{hvm@cage.ugent.be}
\end{tabular}
}
\maketitle

{\footnotesize \texttt{Keywords:} Severi variety, Veronese variety, Segre variety, Grassmann variety, exceptional Tits-building. \\ \texttt{MSC Code:} 51E24, 51A45}

\begin{abstract}
Our main aim is to provide a uniform geometric characterization of the analogues over arbitrary fields of the four complex Severi varieties, i.e.~the quadric Veronese varieties in 5-dimensional projective spaces, the Segre varieties in 8-di\-men\-sional projective spaces, the line Grassmannians in 14-dimensional projective spaces, and the exceptional varieties of type $\mathsf{E}_{6}$ in 26-dimensional projective space. Our theorem can be regarded as a far-reaching generalization of Mazzocca and Melone's approach to finite quadric Veronesean varieties. This approach takes projective properties of complex Severi varieties as smooth varieties as axioms. 
\end{abstract}
\newpage
\section{Introduction}
\subsection{General context of the results}
In the mid fifties, Jacques Tits \cite{TitsThesis} found a new way to approach complex Lie groups by attaching a canonically defined abstract geometry to it. The main observation was that the Dynkin diagram related to the simple complex Lie algebra could be interpreted as an identity card for this geometry. This led Tits to define such geometries over arbitrary fields giving birth to the theory of \emph{buildings}. In the same work \cite{TitsThesis}, Jacques Tits introduced for the first time what Freudenthal would call much later the \emph{Magic Square}, which was in \cite{TitsThesis} a $(4\times 4)$-table  of Dynkin diagrams, each symbolizing a precise geometry that was constructed before, except for the very last entry (the $\mathsf{E}_8$-entry), for which Tits only made some conjectures regarding various dimensions. All these geometries were defined over the complex numbers, but one could easily extend the construction to arbitrary fields, although small characteristics would give some problems. It is essential to note that these geometries were constructed as subgeometries of a projective geometry, and not as abstract geometries. As such, they can be regarded as a kind of realization of the corresponding building (which is the abstract geometry). The geometries of the Magic Square, later better known as the \emph{Freudenthal-Tits Magic Square} \cite{Tits2}, were intensively investigated with tools from algebraic geometry, since they define smooth varieties in complex projective space. One prominent example of this, which is directly related to the present paper, is the classification of the complex Severi varieties by Zak \cite{Zak}, see also Chaput \cite{Chaput} and Lazarsfeld \& Van de Ven \cite{Laz-AVV:84}.  It turns out that the complex Severi varieties correspond exactly to the split geometries of the second row of the Freudenthal-Tits Magic Square (FTMS), namely the quadric Veronese varieties in 5-dimensional projective spaces, the Segre varieties in 8-dimensional projective spaces, the line Grassmannians in 14-dimensional projective spaces, and the exceptional varieties of type $\mathsf{E}_{6}$ in 26-dimensonal projective space. For a recent approach with K-theory, see Nash \cite{Nash}. In the present paper we present a way to approach the geometries of the second row of the FTMS \emph{over any field}. The main idea can be explained by both a bottom-up approach and a top-down approach. 

\begin{itemize}\item \textbf{Bottom-up.} The smallest complex Severi variety or, equivalently, the geometry of the first cell of the second row of the complex FTMS, is the Veronesean of all conics of a complex projective plane. This object can be defined over any field, and a characterization of the finite case by Mazzocca-Melone \cite{Maz-Mel:84} was achieved in the mid-eighties. This characterization was generalized to arbitrary fields by the authors \cite{JSHVM1}. Here the question is whether a further generalization is possible by considering ``quadrics'' in the Mazzocca-Melone axioms instead of conics. A first step was made in \cite{JSHVM2,JSHVM3}, where all quadrics in $3$-space which are not the union of two planes are considered, and the corresponding objects are classified. The axioms below generalize this further to all dimensions, using split quadrics. Hence, our results can be seen as a far-reaching generalization of Mazzocca \& Melone's characterization of the quadric Veronesean in $5$-dimensional finite projective space.  

\item \textbf{Top-down.} Here, the question is how to include the (split) geometries of the second row of the FTMS, defined over arbitrary fields in Zak's result. Thus, one would like to have a characteristic property of these point sets in projective space close to the requirement of being a Severi variety (which, roughly, just means that the secant variety of the smooth non-degenerate complex variety  is not the whole space, and the dimension of the space is minimal with respect to this property). The naturalness of the Mazzocca-Melone axioms alluded to above is witnessed by the fact that they reflect the properties of complex Severi varieties used by Chaput \cite{Chaput} and by  Lazarsfeld \& Van de Ven \cite{Laz-AVV:84} to give an alternative proof of Zak's result.  Note that,  indeed, Zak proves (see Theorem IV.2.4 of \cite{Zak}) that every pair of points of a $2n$-dimensional complex Severi variety is contained in a non-degenerate $n$-dimensional quadric (and no more points of the variety are contained in the space spanned by that quadric). Also, the spaces generated by two of these quadrics intersect in a space entirely contained in both quadrics. These are the (Mazzocca-Melone) properties that we take as axioms, together with the in the smooth complex case obvious fact that the tangent space at a point is (at most) $2n$-dimensional. The latter is achieved by requiring that the tangent spaces to the quadrics through a fixed point are contained in a $2n$-dimensional space. Remarkably, we show in the present paper that these requirements suffice to classify these ``point sets''. 

\end{itemize}

So it is interesting to see how these two points of view meet in our work. There remains to explain the choice of which kind of quadrics. We consider the same class as Zak was dealing with in the complex case: split quadrics (as the complex numbers are algebraically closed). Hence we consider so-called \emph{hyperbolic} quadrics in odd-dimensional projective space, and so-called \emph{parabolic} quadrics in even-dimensional projective space. This way, Zak's result follows entirely from ours, once the immediate consequences of the definition of Severi variety are accepted. 

The complexity and generality of the characterization (remember we work over arbitrary fields) does not allow a uniform proof dealing with all dimensions at the same time (where dimension $d$ means that $d+1$ is the dimension of the ambient projective space of the quadrics; we do not assume anything on the dimension of the whole space). Some small cases for $d$, however, have already been dealt with in the literature in other contexts: The case $d=1$, leading to the quadric Veronesean varieties in \cite{JSHVM1}, the case $d=2$, leading to the Segre varieties in \cite{JSHVM3}, and the case $d=3$, leading to nonexistence, in \cite{JSHVM4}. We consider it the main achievement of this paper to present a general and systematic treatment of the remaining cases, including the intrinsically more involved line-Grassmannians and exceptional varieties of type $\mathsf{E}_{6}$.  

We obtain more than merely a characterization of all geometries of the second row of the FTMS. Indeed, in contrast to Zak's original theorem, we do not fix the dimension of the space we are working in. This implies that we obtain some more ``varieties'' in our conclusion. Essentially, we obtain all subvarieties of the varieties of the second row of the split Freudenthal-Tits Magic Square that are controlled by the diagram of the corresponding building. 

Note that we do not attempt to generalize Zak's result to arbitrary fields in the sense of being secant defective algebraic varieties. Such a result would require scheme theory and does not fit into our general approach to the Freudenthal-Tits Magic Square.  

Let us mention that there is a non-split version of the second row of the FTMS, which consists of the varieties corresponding to the projective planes defined over quadratic alternative \emph{division} algebras. These objects satisfy the same Mazzocca-Melone axioms as we will introduce below, except that the quadrics are not split anymore, but on the contrary have minimal Witt index, namely Witt index $1$. It is indeed shown in \cite{KSVM} that point sets satisfying the Mazzocca-Melone axioms below, but for quadrics of Witt index 1, are precisely the Veronesean representation of the projective planes over composition division algebras (and this includes the rather special case of inseparable field extensions in characteristic 2). In fact, this leads to an even more daring conjecture which, loosely, is the following.
  
\begin{quote}\em
$(*)$ The split and non-split varieties related to the second row of the Freudenthal-Tits Magic Square are characterized by the Mazzocca-Melone axioms, using arbitrary nondegenerate quadrics.
\end{quote}

Conjecture $(*)$ thus consists of three parts, the \emph{split} case,  which uses split quadrics, the \emph{non-split} case, using quadrics of Witt index 1, and the \emph{intermediate} case, where one uses quadrics which have neither maximal nor minimal Witt index. If Conjecture $(*)$ is true, then the intermediate case does not occur.  The other cases are now proved. 


\subsection{The varieties of the second row of the split FTMS}

We now discuss the different geometries that are captured by the second row of the FTMS. 
The split case of the second row of the Freudenthal-Tits Magic Square (FTMS) can be seen as the family of ``projective planes'' coordinatized by the standard split composition algebras $\A$ over an arbitrary field $\K$. These algebras are $\K$, $\K\times\K$, $\mathsf{M}_{2\times2}(\K)$ and $\O'(\K)$, which are the field $\K$ itself, the split quadratic extension of $\K$ (direct product with component-wise addition and multiplication), the split quaternions (or the algebra of $2\times2$ matrices over $\K$), and the split octonions, respectively.   Each of these cases corresponds with a different cell in the second row of the FTMS.

The first cell corresponds with $\A=\K$ and contains the ordinary quadric Veronese variety of the standard projective plane over $\K$, as already mentioned above. Mazzocca and Melone \cite{Maz-Mel:84} characterized this variety in the finite case for fields of odd characteristic using three simple axioms, which we referred to as the \emph{Mazzocca-Melone axioms} above. They can be formulated as follows. Let $X$ be a spanning point set of a projective space, and $\Xi$ a family of planes such that $X\cap\xi$ is a conic for each $\xi\in\Xi$. Then Axiom 1 says that every pair of points of $X$ is contained in a member of $\Xi$; Axiom 2 says that the intersection of two members of $\Xi$ is entirely contained in $X$; Axiom 3 says that for given $x\in X$, the tangents at $x$ to the conics $X\cap\xi$ for which $x\in\xi$ are all contained in a common plane. This characterization has been generalized step-by-step by various authors  \cite{Hir-Tha:91,Tha-Mal:04b}, until the present authors proved it in full generality for Veronese surfaces over arbitrary fields \cite{JSHVM1} (where in the above axioms ``conics'' can be substituted with ``ovals'').

The second cell corresponds with $\A=\K\times\K$ and contains the ordinary Segre variety $\mathcal{S}_{2,2}(\K)$ of two projective planes. In \cite{JSHVM3}, the authors characterize this variety using the above Mazzocca-Melone axioms substituting ``conic'' with ``hyperbolic quadric in 3-dimensional projective space'', ``family of planes'' with ``family of $3$-spaces'', and in Axiom 3 ``common plane'' with ``common $4$-space'', and using a dimension requirement to exclude $\cS_{1,2}(\K)$ and $\cS_{1,3}(\K)$.

In the present paper, we consider the natural extension of these axioms using ``quadric of maximal Witt index (or ``split'' quadric) in $(d+1)$-dimensional space'', ``family of $(d+1)$-dimensional spaces'', and ``common $2d$-space'',  instead of ``conic'', ``family of planes'' and ``common $4$-space'', respectively. The first cell corresponds with $d=1$, the second with $d=2$. Sets satisfying these axioms will be referred to as \emph{Mazzocca-Melone sets of split type $d$}. It will turn out that the third cell corresponds with $d=4$ and the fourth with $d=8$.

The third cell corresponds with $\A=\mathsf{M}_{2\times2}(\K)$ and contains the line Grassmannian variety of projective $5$-space over $\K$.  The fourth cell corresponds with a split octonion algebra. The corresponding geometry is the $1$-shadow of the building of type $\mathsf{E}_6$ over $\K$.
In order to handle the latter case, we need an auxiliary result on varieties related to the half-spin geometries of buildings of type $\mathsf{D}_5$. In fact, that result corresponds with $d=6$. In short, our first Main Result states that Mazzocca-Melone sets of split type $d$ only exist for $d\in\{1,2,4,6,8\}$, and a precise classification will be given in these cases (containing the varieties mentioned above). As a corollary, we can single out the varieties of the second row of the FTMS by a condition on the dimension|namely, the same dimensions used by Zak.

\section{Mazzocca-Melone sets of split type}
\subsection{The axioms}
We now present the precise axioms and in a subsection the main examples, and then in the next section we can state our main results. 

In the definition of our main central object, the notion of a quadric plays an important role.  A quadric in a projective space is the null set of a quadratic form. This is an analytic definition. However, we insist on using only projective properties of quadrics as subsets of points of a projective space. One of the crucial notions we use is the tangent subspace $T_x(Q)$ at a point $x$ of a quadric $Q$. This can be defined analytically, leaving the possibility open that different quadratic forms (different up to a scalar multiple) could define the same point set, or that the tangent space would depend on the chosen basis, which would imply that the tangent space is not well defined. However, it is well know, and very easy to prove that a point $y\neq x$ belongs to the analytically defined tangent space at $x$ to $Q$ if and only if the line $\<x,y\>$ either meets $Q$ in exactly one point (namely, $x$), or fully belongs to $Q$. Alternatively, a point $z$ does not belong to the tangent space at $x$ if and only if the line $\<x,y\>$ intersects $Q$ in precisely $2$ points.  So $T_x(Q)$ is defined by the point set $Q$. 

Now let $X$ be a spanning point set of $\PAG^{N}(\mathbb{K})$, $N\in\N\cup\{\infty\}$, with $\mathbb{K}$ any field, and let $\Xi$ be a collection of $(d+1)$-spaces of $\PAG^{N}(\mathbb{K})$, $d\geq 1$, such that, for any $\xi\in\Xi$, the intersection $\xi\cap X=:X(\xi)$ is a non-singular split quadric (which we will call a \emph{\EQ}, inspired by the theory of parapolar spaces, see~\cite{Shu:12}; for $d=1$, we sometimes say \emph{conic})
in $\xi$ (and then, for $x\in X(\xi)$, we denote the tangent space at $x$ to $X(\xi)$ by $T_x(X(\xi))$ or sometimes simply by $T_x(\xi)$). We call $(X,\Xi)$ a \emph{Mazzocca-Melone set (of split type $d$)} if (MM1), (MM2) and (MM3) below hold.

The condition $d\geq 1$ stems from the observation that, if we allowed $d\in\{-1,0\}$, then we would only obtain trivial objects (for $d=-1$, a single point; for $d=0$ a set of points no 3 of which collinear and no 4 of which co-planar).

A Mazzocca-Melone set is called \emph{proper} if $|\Xi|>1$. Non-proper Mazzocca-Melone sets of split type are just the split quadrics themselves (we use the expression ``of split type'' to leave the exact number $d$ undetermined, but still to indicate that the quadrics we use are split). Also, the members of $\Xi$ are sometimes called \emph{the quadratic spaces}.

\begin{itemize}
\item[(MM1)] Any pair of points $x$ and $y$ of $X$ lies in at least one element of $\Xi$, denoted by $[x,y]$ if unique.

\item[(MM2)] If $\xi_1,\xi_2\in \Xi$, with $\xi_1\neq \xi_2$, then $\xi_1\cap\xi_2\subset X$.

\item[(MM3)] If $x\in X$, then all $d$-spaces $T_x(\xi)$, $x\in\xi\in\Xi$, generate a subspace $T_x$  of $\PAG^{N}(\mathbb{K})$ of dimension at most $2d$.

\end{itemize}

The central problem of this paper is to classify all Mazzocca-Melone sets of split type $d$, for all $d\geq 1$. We will state this classification in Section~\ref{sec:main result}, after we have introduced the examples in the next section.

We end this introduction by mentioning that the first case of non-existence, namely $d=3$, was proved in \cite{JSHVM4} in the context of the first cell of the third row of the FTMS. So we can state the following proposition.

\begin{prop}\label{d=3}
A proper Mazzocca-Melone set of split type $3$ does not exist.
\end{prop}

\begin{proof}
This is basically Proposition 5.4 of \cite{JSHVM4}. That proposition states that no Lagrangian set of diameter 2 exists. A Lagrangian set $(X,\Xi)$ satisfies the same axioms as a Mazzocca-Melone set of split type 3, except that (MM1) is replaced by (MM1$''$) below. The diameter of a Langrangian set is defined as the diameter of the following graph $\Gamma$. The vertices are the elements of $X$ and two vertices are adjacent if they are contained in a singular line of $X$ (a line fully contained in $X$). 
\begin{itemize}
\item[(MM1$''$)] Any pair of points $x$ and $y$ of $X$ at distance at most 2 in the graph $\Gamma$ is contained in at least one element of $\Xi$.
\end{itemize}
By Axiom (MM1), any Mazzocca-Melone set of split type 3 has diameter $2$ and hence is a Langangian set of diameter 2, whic does not exist by Proposition 5.4 of \cite{JSHVM4}.
\end{proof}

As a side remark, we note that a similar conjecture as above for the third row can be stated,  supported by a recent result of  De Bruyn and the second author \cite{BDB-HVM}. Some other work that is related to our axiomatic approach, but still with a restriction on the underlying field, has been carried out by Russo \cite{Rus:09} and by Ionescu \& Russo \cite{Ion-Rus:08,Ion-Rus:10}.

\subsection{Examples of Mazzocca-Melone sets of split type}\label{examples}


We define some classes of varieties over the arbitrary field $\K$. Each class contains Mazzocca-Melone sets of split type. This section is meant to introduce the notation used in the statements of our results.  

\textbf{Quadric Veronese varieties |}
The \emph{quadric Veronese variety} $\mathcal{V}_{n}(\K)$, $n\geq 1$, is the set of points in $\PAG^{{n+2 \choose 2} - 1}(\K)$
obtained by taking the images of all points of $\PAG^n(\K)$ under the Veronese mapping, which maps the point  $(x_{0},\cdots,x_{n})$ of $\PAG^n(\K)$ to the point $(x_{i}x_{j})_{0\leq i\leq j \leq n}$ of $\PAG^{{n+2 \choose 2} - 1}(\K)$. If $\K=\C$, then this is a smooth non-degenerate complex algebraic variety of dimension $n$.

\textbf{Line Grassmannian varieties |}
The \emph{line Grassmannian variety} $\mathcal{G}_{m,1}(\K)$, $m\geq 2$, of $\PAG^{m}(\K)$ is the set of points of $\PAG^{\frac{m^2+m-2}{2}}(\K)$ obtained by taking the images of all lines of $\PAG^{m}(\K)$ under the Pl\"ucker map $$\rho (\<(x_0,x_1,\ldots,x_m),(y_0,y_1,\ldots,y_m)\>)=\left(\left|\begin{array}{cc} x_i & x_j\\ y_i & y_j\end{array}\right|\right)_{0\leq i< j\leq m}.$$ If $\K=\C$, then this is  a smooth non-degenerate complex algebraic variety of dimension $2m-2$.

\textbf{Segre varieties |}
The \emph{Segre variety} $\mathcal{S}_{k,\ell}(\K)$ of $\PAG^{k}(\K)$ and $\PAG^{\ell}(\K)$ is the set of points of $\PAG^{k\ell+k+\ell}(\K)$ obtained by taking the images of all pairs of points, one in $\PAG^{k}(\K)$ and one in $\PAG^{\ell}(\K)$, under the Segre map $$\sigma (\<(x_0,x_1,\ldots,x_k),(y_0,y_1,\ldots,y_\ell)\>)=(x_iy_j)_{0\leq i\leq k;0\leq j\leq\ell}.$$
If $\K=\C$, then this is a smooth non-degenerate complex variety of dimension $k+\ell$.

\textbf{Varieties from split quadrics |}
Every split quadric is by definition a variety, which we will refer to as the \emph{variety corresponding to a split quadric}. The one corresponding to a \emph{parabolic quadric} in $\PAG^{2n}(\K)$ will be denoted by $\mathcal{B}_{n,1}(\K)$; the one corresponding to a \emph{hyperbolic quadric} in $\PAG^{2n-1}(\K)$ by $\mathcal{D}_{n,1}(\K)$.

\textbf{Half-spin varieties |}
The exposition below is largely based on \cite{IM}, but the results are due to Chevalley \cite{Chevalley}, see also the recent reference \cite{Manivel}.

Let $V$ be a vector space of dimension $2n$ over a field $\mathbb{K}$ with a nonsingular quadratic form $q$ of maximal Witt index giving rise to a hyperbolic quadric and $(\cdot,\cdot)$ the associated bilinear form.
Then the maximal singular subspaces of $V$ with respect to $q$ have dimension $n$ and fall into two classes $\Sigma^{+}$ and $\Sigma^{-}$ so that two subspaces have an intersection of even codimension if and only if they are of the same type.

Then we define the half spin geometry $\mathcal{D}_{n,n}(\K)$ as follows. The point set $\mathcal{P}$ is the set of maximal totally singular subspaces of $V$ of one particular type, say +. For each totally singular $(n-2)$-space $U$, form a line by taking all the points containing $U$. The line set $\mathcal{L}$ is the collection of sets of this form.

Fix a pair of maximal isotropic subspaces $U_0$ and $U_\infty$ such that $V=U_0\oplus U_\infty$.
Let $S=\bigwedge U_\infty$ be the exterior algebra of $U_\infty$, called the \emph{spinor space of $(V,q)$} and the even and odd parts
of $S$ are called the \emph{half-spinor spaces}, $S^+=\bigwedge^{\mathrm{even}} U_\infty$, $S^-=\bigwedge^{\mathrm{odd}} U_\infty$.

To each maximal isotropic subspace $U\in \Sigma^+\cup \Sigma^-$ one can associate a unique, up to proportionality, nonzero half-spinor $s_U \in S^+\cup S^-$ such that $\phi_u(s_U)=0$, for all $u\in U$, where $\phi_u \in \mbox{ End}(S)$ is the Clifford automorphism of $S$ associated to $U$:
$$\phi_u(v_1\wedge \cdots \wedge v_k)=\sum_i (-1)^{i-1} (u_0,v_i)v_1\wedge \cdots \wedge \hat{v_{i}} \wedge \cdots \wedge v_k+
u_\infty\wedge v_1 \wedge \cdots \wedge v_k,$$
where $u=u_0+u_\infty$, with $u_o\in U_0, u_\infty\in U_\infty$.

One can also obtain an explicit coordinate description, as well as a set of defining quadratic equations for the point set, see e.g. \cite{Chevalley,IM,Manivel,Mukai}. 

We shall refer to these varieties as the \emph{half-spin varieties $\mathcal{D}_{n,n}(\K)$.}


\textbf{The variety related to a building of type $\mathsf{E}_6$ |}
Let $V$ be a 27-dimensional vector space over a field $\mathbb{K}$ consisting of all ordered triples
$x=[x^{(1)},x^{(2)},x^{(3)}]$ of $3\times 3$ matrices $x^{(i)},1\leq i\leq 3$, where addition and scalar multiplication
are defined entrywise. Moreover the vector space is equipped with the cubic form $d:V\to \mathbb{K}$ given by
$$d(x)=\det x^{(1)}+\det x^{(2)}+\det x^{(3)}-\operatorname{Tr} (x^{(1)}x^{(2)}x^{(3)})$$
for $x\in V$,  see for instance Aschbacher's article \cite{Asch1}.

Let $x\in V$. The map $\partial_x d: V\to \mathbb{K}$ such that $\partial_x d(y)$ is the coefficient in the expansion of $d(x+ty)\in \mathbb{K}[t]$ as a polynomial in $t$ is called the \emph{derivative} of $d$ at $x$.
Define the adjoint square $x^\sharp$ of $x$ by $\partial_x(d)(y)=\operatorname{Tr}(x^\sharp y)$.
The \emph{variety $\mathcal{E}_{6,1}(\K)$} consists of the set of projective points $\langle x \rangle$ of $V$ for which $x^\sharp=0$.

\section{Main result}\label{sec:main result}

We can now state our main result.

\textbf{Main Result.} \emph{A proper Mazzocca-Melone set of split type $d\geq 1$ in $\PAG^{N}(\K)$  is projectively equivalent to one of the following:
\begin{itemize}
\item[$d=1$]\begin{itemize}
\item the quadric Veronese variety $\mathcal{V}_{2}(\K)$, and then $N=5$;
\end{itemize}
\item[$d=2$]\begin{itemize}
\item the Segre variety $\mathcal{S}_{1,2}(\K)$, and then $N=5$;
\item the Segre variety $\mathcal{S}_{1,3}(\K)$, and then $N=7$;
\item the Segre variety $\mathcal{S}_{2,2}(\K)$, and then $N=8$;
\end{itemize}
\item[$d=4$]
\begin{itemize}
\item the line Grassmannian variety $\mathcal{G}_{4,1}{(\K)}$, and then $N=9$;
\item the line Grassmannian variety $\mathcal{G}_{5,1}{(\K)}$, and then $N=14$;
\end{itemize}
\item[$d=6$]
\begin{itemize}
\item the half-spin variety $\mathcal{D}_{5,5}(\K)$, and then $N=15$;
\end{itemize}
\item[$d=8$]
\begin{itemize}
\item the variety $\mathcal{E}_{6,1}(\K)$, and then $N=26$;
\end{itemize}
\end{itemize}
}

\begin{rem}\rm If one includes the non-proper cases in the previous statement, then a striking similarity between these and the proper cases becomes apparent (and note that each symp of any proper example in the list is isomorphic to the non-proper Mazzocca-Melone set of the same split type). Indeed, for $d=1$, a conic is a Veronesean variety $\cV_1(\K)$; for $d=2$, a hyperbolic quadric in $3$-space is a Segre variety $\cS_{1,1}(\K)$;  for $d=4$, a hyperbolic quadric in $5$-space is the line Grassmannian variety $\mathcal{G}_{3,1}(\K)$; for $d=6$, the hyperbolic quadric in $7$-space is, by triality, also the half-spin variety $\mathcal{D}_{4,4}(\K)$; finally for $d=8$, the hyperbolic quadric in $9$-space is a variety $\mathcal{D}_{5,1}$, sometimes denoted $\mathcal{E}_{5,1}$ to emphasize the similarity between objects of type $\mathsf{D_5}$ and the ones of type $\mathsf{E_6},\mathsf{E_7}$ and $\mathsf{E_8}$. 
\end{rem}

The following corollary characterizes precisely the varieties related to the second row of the FTMS. It is obtained from our Main Result by either adding a restriction on the global dimension, or a restriction on the \emph{local} dimension, i.e., the dimension of at least one tangent space.

\textbf{Main Corollary.}  \emph{A Mazzocca-Melone set of split type $d$, with $d\geq 1$, in $\PAG^{N}(\K)$, with either $N\geq 3d+2$, or $\dim T_x=2d$ for at least one $x\in X$,  is projectively equivalent to one of the following:
\begin{itemize*}
\item the quadric Veronese variety $\mathcal{V}_{2}(\K)$, and then $d=1$;
\item the Segre variety $\mathcal{S}_{2,2}(\K)$, and then $d=2$;
\item the line Grassmannian variety $\mathcal{G}_{5,1}{(\K)}$, and then $d=4$;
\item the variety $\mathcal{E}_{6,1}(\K)$, and then $d=8$.
\end{itemize*}
In all cases $N=3d+2$ and $\dim T_x=2d$ for all $x \in X$.
}

\bigskip

Concerning the proof, the cases $d=1$, $d=2$  and $d=3$ of the Main Result are proved in \cite{JSHVM1}, \cite{JSHVM3} and Proposition~\ref{d=3}, respectively. So we may suppose $d\geq 4$. However, our proof is inductive, and in order to be able to use the cases $d\in\{1,2,3\}$, we will be forced to prove some results about sets only satisfying (MM1) and (MM2), which we will call \emph{pre-Mazzocca-Melone sets (of split type)}, and this will include $d=1,2,3$.



The rest of the paper is organized as follows. In the next section, we prove some general results about (pre-)Mazzocca-Melone sets, and use these in Section~\ref{exist} to finish the proof of the Main Result for all the cases corresponding with the varieties in the conclusion of the Main Result. In Section~\ref{notexist}, we prove the non-existence of proper Mazzocca-Melone sets for $d\notin\{1,2,4,6,8\}$, and in the last section, we verify the axioms. 

\section{General preliminary results for pre-Mazzocca-Melone sets}\label{sec:Segre}

We introduce some notation. Let $(X,\Xi)$ be a pre-Mazzocca-Melone set of split type $d$ in $\PAG^{N}(\K)$, $d\geq 1$. Axiom (MM1) implies that, for a given line $L$ of $\PAG^{N}(\K)$, either $0$, or $1$, or $2$ or all points of $L$ belong to $X$, and in the latter case $L$ is contained in a \EQ. In this case, we call $L$ a \emph{singular line}. More generally, if all points of a  $k$-space of $\PAG^{N}(\K)$ belong to $X$, then we call this $k$-space \emph{singular}. Two points of $X$ contained in a common singular line will be called \emph{$X$-collinear}, or simply \emph{collinear}, when there is no confusion. Note that there is a unique \EQ\ through a pair of non-collinear points (existence follows from (MM1) and uniqueness from (MM2)). A \emph{maximal} singular subspace is one that is not properly contained in another.

The linear span of a set $S$ of points in $\PAG^{N}(\K)$ will be denoted by $\<S\>$.

We now have the following lemmas.

\begin{lemma}[The Quadrangle Lemma] Let $L_1,L_2,L_3,L_4$ be four (not necessarily pairwise distinct) singular lines such that $L_i$ and $L_{i+1}$ (where $L_{5}=L_{1}$) share a (not necessarily unique) point $p_{i}$, $i=1,2,3,4$, and suppose that $p_1$ and $p_3$ are not $X$-collinear. Then $L_1,L_2,L_3,L_4$ are contained in a unique common \EQ.
\end{lemma}

\begin{proof}
Since $\<p_1,p_3\>$ is not singular, we can pick a point $p\in\<p_1,p_3\>$ which does not belong to $X$. Let $\xi$ be the unique \EQ\ containing $p_1$ and $p_3$. We choose two arbitrary  but distinct lines $M_2,M_3$ through $p$ inside the plane $\<L_2,L_3\>$. Denote $M_i\cap L_j=\{p_{ij}\}$, $\{i,j\}\subseteq\{2,3\}$. By (MM1) there is a \EQ\ $\xi_i$ containing $p_{i2}$ and $p_{i3}$, $i=2,3$.   If $\xi_2\neq \xi_3$, then (MM2) implies that $p$, which is contained in $\xi_2\cap\xi_3$, belongs to $X$, a contradiction. Hence $\xi_2=\xi_3=\xi$ and contains $L_2,L_3$. We conclude $\xi$ contains $L_2,L_3$, and similarly also $L_4,L_1$.
\end{proof}

Note that, if in the previous lemma all lines are different, then they span a $3$-space; this just follows from the lemma a posteriori.

\begin{lemma}\label{subspace}
Let $p\in X$ and let $H$ be a symp not containing $p$. Then the set of points of $H$ collinear with $p$  is either empty or constitutes a singular subspace of $H$.
\end{lemma}

\begin{proof}
Suppose first that $p$ is $X$-collinear with two non-collinear points $x_1,x_2\in H$, and let $x\in H$ be collinear with both $x_1$ and $x_2$. Then the Quadrangle Lemma applied to $p,x_1,x,x_2$ implies that $p\in H$, a contradiction. Hence all points of $H$ collinear with $p$ are contained in a singular subspace. Suppose now that $p$ is collinear with two collinear points $y_1,y_2\in H$ and let $y$ be a point on the line $\<y_1,y_2\>$. If we assume that $p$ and $y$ are not collinear, then the Quadrangle Lemma applied to $p,y_1,y,y_2$ yields that the plane $\<p,y_1,y_2\>$ is entirely contained in a symp, contradicting the non-collinearity of $p$ and $y$. Hence $p$ and $y$ are collinear and the lemma is proved.
\end{proof}

\begin{rem}\rm Note that, up to now, we did not use the part of (MM1) that says that two \emph{collinear} points of $X$ are contained in some symp. In fact, this follows from the previous lemmas. Indeed, let $x,y$ be two $X$-collinear points and suppose that they are not contained in any common symp. By considering an arbitrary symp, we can, using Lemma~\ref{subspace}, find a point not collinear to $y$ and hence a symp $H$ containing $y$ (but not $x$, by assumption). Let $z$ be a point of $H$ collinear with $y$ but not collinear with $x$ (and $z$ exists by Lemma~\ref{subspace}). The Quadrangle Lemma applied to $x,y,z,y$ now implies that  the (unique) symp $X([x,z])$ contains $y$. Hence we can relax (MM1) by restricting to pairs of points that are not $X$-collinear. For ease of reference, we have not included this minor reduction in the axioms.
\end{rem}

\begin{lemma} \label{lemma1} A pair of singular $k$-spaces, $k\geq 0$, that intersect in a singular $(k-1)$-space is either contained in a \EQ, or in a singular $(k+1)$-space. In particular, if $k>\lfloor \frac{d}{2} \rfloor$, then such a pair is always contained in a singular $(k+1)$-space.
\end{lemma}

\begin{proof} If $k=0$, then this follows from (MM1). If $k=1$, this follows from the Quadrangle Lemma. Now let $k\geq 2$. Suppose $A,B$ are singular $k$-spaces intersecting in a singular $(k-1)$-space and suppose that $C:=\<A,B\>$ is not singular. Then there is a point $p\in C$ which does not belong to $X$. Choose a point $p_A\in A\setminus B$ and let $p_B$ be the intersection of $\<p,p_A\>$ and $B$. Then there is a unique \EQ\ $\xi$ containing $p_A$ and $p_B$. Let $q_A$ be any other point of $A\setminus B$ and put $q=\<p_A,q_A\>\cap B$. Then the Quadrangle Lemma (with $L_1=L_2$ and $L_3=L_4$) implies that $q_A\in\xi$. Hence $A\subseteq \xi$ and similarly $B\subseteq\xi$.
 \end{proof}

\begin{lemma}\label{lemma2} Every singular $k$-space, $k\leq \lfloor \frac{d}{2} \rfloor$, which is contained in a finite-dimensional maximal singular subspace, is contained in a \EQ.
\end{lemma}

\begin{proof}Clearly, the lemma is true for $k=0,1$. Now let $k\geq 2$.

Let $A$ be a singular $k$-space contained in the finite-dimensional maximal singular subspace $M$. Note that, in principle, the dimension of $M$ should not necessarily be larger than $\lfloor \frac{d}{2} \rfloor$. Let $S$ be a subspace of $A$ of maximal dimension with the property that it is contained in some \EQ\ $H$. We may assume $S\neq A$, as otherwise we are done. 
In $H$, the subspace $S$ is not a maximal subspace as this would imply that $\lfloor \frac{d}{2} \rfloor = \dim S<\dim A=k\leq \lfloor \frac{d}{2} \rfloor$, a contradiction. Hence, since $H$ is a polar space, we can find two singular subspaces $B_1$ and $B_2$ of $H$, containing $S$, having dimension $\dim S +1$ and such that $B_1$ and $B_2$ are not contained in a common singular subspace of $H$. It then follows that $B_1$ and $B_2$ are not contained in a common singular subspace of $X$, as $\<B_1,B_2\>\subseteq \<H\>$ and since $H=\<H\>\cap X$, this would imply that $\<B_1,B_2\>$ is a singular subspace of $H$, a contradiction. Hence at most one of $B_1,B_2$ belongs to $A$, and so we may assume that $B_1$ is not contained in $A$. Consequently, $B_1$ is a singular subspace with dimension $\dim(S)+1$ contained in $H$, containing $S$ and not contained in $M$. Put $\ell=\dim(M)-\dim(A)+1$ and let $A_1\subseteq A_2\subseteq \cdots\subseteq A_\ell$ be a family of nested subspaces, with $\dim(A_i)=\dim(S)+i$, for all $i\in\{1,2,\ldots,\ell\}$, with $A_\ell=M$ and $S\subseteq A_1\subseteq A$. Put $B_i=\<A_{i-1},B_1\>$, $i=2,3,\ldots,\ell+1$. By the maximality of $S$, Lemma~\ref{lemma1} implies that $B_2$ is a singular subspace. Let  $2\leq i\leq\ell$. If $B_i$ is singular, then, since $B_i\cap A_i=A_{i-1}\supseteq A_1$ and $A_1$ is not contained in a symp, Lemma~\ref{lemma1} again implies that $\<A_i,B_i\>=\<A_i,B_1\>=B_{i+1}$ is a singular subspace. Inductively, this implies that $B_{\ell+1}$, which properly contains $M$, is a singular subspace. This contradicts the maximality of $M$. Hence $S=A$ and the assertion follows. 
\end{proof}

The next lemma suggests an inductive approach.

\begin{lemma}\label{induction}
Suppose $(X,\Xi)$ is a pre-Mazzocca-Melone set of split type $d$, $d\geq 4$. Let $x\in X$ be arbitrary and assume that $T_x$ is finite-dimensional. Let $C_x$ be a subspace of $T_x$ of dimension $\dim(T_x)-1$ not containing $x$. Consider the set $X_x$ of points of $C_x$ which are contained in  a singular line of $X$ through $x$. Let $\Xi_x$ be the collection of subspaces of $C_x$ obtained by intersecting $C_x$ with all $T_x(\xi)$, for $\xi$ running through all \EQ s of $X$ containing $x$. Then $(X_x,\Xi_x)$ is a pre-Mazzocca-Melone set of split type $d-2$.
\end{lemma}

\begin{proof} We start by noting that the dimension of a maximal singular subspace in $(X_x,\Xi_x)$ is bounded above by $\dim T_x-1$. Hence we may use Lemma~\ref{lemma2} freely.

Now, Axiom~(MM1) follows from Lemma~\ref{lemma1} (two lines through $x$ not in a singular plane are contained in a \EQ) and Lemma~\ref{lemma2} (a singular plane is contained in a \EQ).

Axiom~(MM2) is a direct consequence of the validity of Axiom (MM2) for $(X,\Xi)$.
\end{proof}

The pair $(X_x,\Xi_x)$ will be called the \emph{residue at $x$}. This can also be defined for split type $3$ in the obvious way, but we do not necessarily obtain a pre-Mazzocca-Melone set of split type 1, as the following shows (the proof is the same as the proof of the previous lemma, except that we do not need Lemma~\ref{lemma2} anymore by the weakening of Axiom~(MM1)). 

\begin{lemma}\label{residuesplittype3}
Suppose $(X,\Xi)$ is a pre-Mazzocca-Melone set of split type $3$. Let $x\in X$ be arbitrary and let $(X_x,\Xi_x)$ be the residue at $x$. Then $(X_x,\Xi_x)$ satisfies {\em Axiom~(MM2)} for split type 1 (so for each $\xi\in \Xi_x$ we have that $X(\xi)$ is a plane conic), and the following weakened version of {\em Axiom~(MM1)}: 
\begin{itemize}
\item[\emph{(MM1$'$)}] Any pair of points $x$ and $y$ of $X$ either lies on a singular line, or is contained in a unique member of $\Xi_x$.
\end{itemize}
If in particular $(X,\Xi)$ has no singular planes, then $(X_x,\Xi_x)$ is a pre-Maxxocca-Melone set of split type $1$. \qed
\end{lemma} 

In general, we also want the residue to be a proper pre-Mazzocca-Melone set whenever the original is proper.
This follows from the next lemma, but we state a slightly stronger property, which we will also use in other situations.

\begin{lemma}\label{proper}
Let $(X,\Xi)$ be a proper Mazzocca-Melone set of split type $d\geq 1$, and let $x,y\in X$ be collinear. Then there is a \EQ\ containing $x$ and not containing $y$. In particular, the residue at any point of a proper Mazzocca-Melone set of split type $d\geq 4$ is a proper pre-Mazzocca-Melone set of split type $d-2$. If $d=3$, then the residue contains at least two conics.
\end{lemma}

\begin{proof}
It is easy to see that there are at least two different \EQ s $H,H'$ containing $x$ (this follows straight from the properness of $(X,\Xi)$, and, in fact, this already implies that the residue at $x$ is proper).  Suppose both contain $y$. Select a point $p\in H\setminus H'$ collinear to $x$ but not collinear to $y$, and a point $p'\in H'\setminus H$ collinear to $x$ but not collinear to $p$ (these points exist by Lemma~\ref{subspace} and since $H\cap H'$ is a singular subspace). By Lemma~\ref{lemma1} the unique \EQ\ $H''$ containing $p$ and $p'$ also contains $x$. It does not contain $y$, however, because the intersection $H\cap H''$ would otherwise not be a singular subspace, as it would contain the non-collinear points $p$ and $y$.
\end{proof}

In order to be able to use such an inductive strategy, one should also have that Axiom (MM3) is inductive. To obtain this, one has to apply different techniques for the different small values of $d$ (for $d=1$ and $d=2$, there was not even such an induction possible). Roughly, the arguments are different for $d=3,4,5$, there is a general approach for $6\leq d\leq 9$ and another one again for $d\geq 10$.

From now on, we also assume (MM3). The basic idea for the induction  for larger $d$ is to prove that the tangent space $T_x$ does not contain any point $y\in X$ with $\<x,y\>$ not singular.   Such a point $y$ will be called a \emph{wrinkle (of $x$)}, and if $x$ does not have any wrinkles, then $x$ is called \emph{smooth}. Once we can prove that all points are smooth, we get control over the dimension of the tangent spaces in the residues at the points of $X$. For $d=3,4,5$, this fails, however, and we have to use additional arguments to achieve this. The case $d=3$ is ``extra special'', since we cannot even apply Lemma~\ref{induction}. This case will be the most technical of all (but also the case $d=2$ is rather technical, see~\cite{JSHVM3}). 

Concerning wrinkles, we prove below a general lemma that helps in ruling them out. Before that, we need two facts about hyperbolic and parabolic quadrics. In passing, we also prove a third result which we will need later. We freely use the well known fact that, if for any quadric $Q$ and any point $x\in Q$, some line $L$ through $x$ does not belong to the tangent space of $Q$ at $x$, then $L$ intersects $Q$ in precisely two points, one of which is $x$.

\begin{lemma}\label{quadric}
Let $H$ be a hyperbolic quadric in $\PAG^{2n+1}(\K)$, $n\geq 1$. Then every $(n+1)$-space of $\PAG^{2n+1}(\K)$ contains a pair of non-collinear points of $H$.
\end{lemma}\begin{proof}
Let, by way of contradiction, $U$ be a subspace of dimension $n+1$ intersecting $H$ precisely in a singular subspace $W$. Since $U$ meets every subspace of dimension $n$ that is completely contained in $H$ in at least one point, we deduce that $W$ is nonempty. Let $x\in W$ be an arbitrary point. 
Since no line in $U$ through $x$ intersects $H$ in precisely two points, each such line is contained in the tangent space $T_x(H)$ of $H$ at $x$. Now, $T_x(H)\cap H$ is a cone over $x$ of a hyperbolic quadric $H'$ in some $\PAG^{2n-1}(\K)$, and $U$ can be considered as a cone over $x$ of an $n$-space in $\PAG^{2n-1}(\K)$ intersecting $H'$ in some subspace. An obvious induction argument reduces the lemma now to the case $n=1$, which leads to a contradiction, proving the lemma.
\end{proof}

\begin{lemma}\label{paraquad}
Let $P$ be a parabolic quadric in $\PAG^{2n}(\K)$, $n\geq 1$. Then through every singular $(n-1)$-space  there exists exactly one $n$-space containing no further points of $P$. Also, every $(n+1)$-space of $\PAG^{2n}(\K)$ contains a pair of non-collinear points of $P$.
\end{lemma}
\begin{proof}
To prove the first assertion, let $W$ be a singular $(n-1)$-space of $P$ and let $x\in W$. Then, as in the previous proof,  $T_x(P)\cap P$ is a cone over $x$ of a parabolic quadric $P'$ in some $\PAG^{2n-2}(\K)$, and every $n$-space of $\PAG^{2n}(\K)$ intersecting $P$ in exactly $W$ corresponds to an $(n-1)$ space of $\PAG^{2n-2}(\K)$ intersecting $P'$ in exactly $W'$, where $W'$ corresponds to $W$. An obvious induction argument reduces the assertion to the case $n=1$, where the statement is obvious (there is a unique tangent in every point).

The second assertion is proved in exactly the same way as  Lemma~\ref{quadric}, again reducing the problem to $n=1$, where the result is again obvious (since $n+1=2n$ in this case).
\end{proof}

\begin{lemma}\label{n-2}
Let $Q$ be a hyperbolic or parabolic quadric in $\PAG^{n}(\K)$, $n\geq 3$, and let  $S$ be a subspace of dimension $n-2$ of $\PAG^n(\K)$. Then some line of $Q$ does not intersect $S$.  
\end{lemma}
\begin{proof}
Pick two disjoint maximal singular subspaces $M_1,M_2$ of $Q$. We may assume that $S$ intersects every maximal singular subspace $M$ in either a hyperplane of $M$, or in $M$ itself, as otherwise we easily find a line inside $M$ not intersecting $S$.  If $Q$ is hyperbolic, then, as $\<M_1,M_2\>=\PAG^n(\K)$, this easily implies that $S=\<S\cap M_1,S\cap M_2\>$ and $\dim (S\cap M_i)=\frac{n-3}{2}$. 
Since collinearity induces a duality between disjoint maximal singular subspaces, there is a unique point in $M_1$ not collinear to any point of $M_2\setminus S$. It follows that we find a point $x_1\in M_1\setminus S$ collinear in $Q$ to some point $x_2\in M_2\setminus S$ and that the line $L$ joining $x_1,x_2$ does not intersect $S$. If $Q$ is parabolic, then the same argument and conclusion holds if $\dim(S\cap\<M_1,M_2\>)=n-3$. If not, then $S\subseteq\<M_1,M_2\>$, and we find a point $x\in M_1\cup M_2\setminus S$ (otherwise $S=\<M_1,M_2\>$, which is $(n-1)$-dimensional, a contradiction). Since not all points of $M_1$ and $M_2$ are collinear to $x$, and since the points collinear to $x$ in $Q$ span a hyperplane of $\PAG^n(\K)$, we find a point $y\in Q$ outside $\<M_1,M_2\>$ collinear to $x$. The line $\<x,y\>$ belongs to $Q$ and does not meet $S$. 
\end{proof}

We are now ready to prove a result that restricts the possible occurrences of wrinkles. It is one of the fundamental observations in our proof.

\begin{lemma}\label{badpoints}
Let $(X,\Xi)$ be a Mazzocca-Melone set of split type $d\geq 1$. Let $x\in X$. Then no wrinkle $y$ of $x$ is contained in the span of two tangent spaces $T_x(\xi_1)$ and $T_x(\xi_2)$, for two different $\xi_1,\xi_2\in\Xi$.
\end{lemma}

\begin{proof}
Suppose, by way of contradiction, that $T_x$ contains the wrinkle $y$, and that there are two \EQ s $H_1$ and $H_2$ such that $x\in H_1\cap H_2$ and $y\in\<T_x(H_1),T_x(H_2)\>$. Then there are points $a_i\in T_x(H_i)$, $i\in\{1,2\}$, such that $y\in a_1a_2$. Our aim is to show that we can (re)choose the wrinkle $y$ and the point $a_1$ in such a way that $a_1\in X$. Considering a \EQ\ through $y$ and $a_1$ and using Axiom (MM2) then implies that $a_2\in X$, and so the plane $\<x,a_1,a_2\>$, containing two singular lines and an extra point $y\in X$, must be singular, contradicting the fact that $y$ is a wrinkle of $x$.

Set $U=T_x(H_1)\cap T_x(H_2)$ and $\dim U=\ell$, $0\leq\ell\leq\lfloor \frac{d}{2} \rfloor$.

By assumption $\dim\<T_x(H_1),T_x(H_2)\>=2d-\ell$. Hence $\dim([x,y]\cap\<T_x(H_1),T_x(H_2)\>)\geq d+1-\ell$ and $\dim(T_x([x,y])\cap\<T_x(H_1),T_x(H_2)\>)\geq d-\ell$. The latter implies that we can find a subspace $W$ of dimension $\lceil \frac{d}{2} \rceil-\ell$ through $x$ contained in $T_x([x,y])\cap \<T_x(H_1),T_x(H_2)\>$, but intersecting $T_x(H_1)\cup T_x(H_2)$ exactly in $\{x\}$ (using the fact that $T_x([x,y])$ intersects $T_x(H_i)$ in a subspace of dimension at most $\lfloor \frac{d}{2} \rfloor$). Note that $W$ is not necessarily a singular subspace. We consider the space $\Pi=\<W,y\>$ of dimension $\lceil \frac{d}{2} \rceil+1-\ell$. Every line in $\Pi$ through $x$ outside $W$ contains a unique wrinkle.   It follows that $\Pi\cap\<H_i\>=\{x\}$, $i\in\{1,2\}$. This is already a contradiction for $d=1$ and $\ell=0$, since this implies $3=2+1=\dim\Pi+\dim\<T_{x}(H_1)\>\leq 2d-\ell=2$. So we assume $d>1$. Using straightforward dimension arguments, we deduce that there are unique $(\lceil \frac{d}{2} \rceil+1)$-spaces $U_i\subseteq\<H_i\>$, $i=1,2$, containing $U$ such that $\Pi\subseteq \<U_1,U_2\>$ (and $\dim\<U_1,U_2\>=2\lceil \frac{d}{2} \rceil+2-\ell$). Let $U_1'$ be the $\lceil \frac{d}{2} \rceil$-space obtained by intersecting $\<U_2,W\>$ with $U_1$. Then, by Lemmas~\ref{quadric} and~\ref{paraquad}, we can pick a point $a_1\in (X\cap U_1)\setminus U_1'$. Since $U_2$ and $\Pi$ meet in only $x$, and $a_1\in\<U_2,\Pi\>$, there is a unique plane $\pi$ containing $x,a_1$ and intersecting both $\Pi$ and $U_2$ in (distinct) lines. By our choice of $a_1$ outside $U_1'$, the line $\pi\cap\Pi$ is not contained in $W$, hence contains a wrinkle, which we may assume without loss of generality to be $y$. Inside the plane $\pi$, the line $a_1y$ intersects $\pi\cap U_2$ in a point $a_2\in \<H_2\>$. 
This completes the proof of the lemma.
\end{proof}

We can show the power of the previous lemma by the following lemma, which establishes the non-existence of wrinkles for small values of $d$.

\begin{lemma}\label{generation}
If $2\leq d\leq 9$, then for every non-smooth point $x\in X$, there exist at least one wrinkle $y$ of $x$ and two \EQ s $H_{1}$ and $H_{2}$ through $x$ such that $y\in \langle T_x(H_{1}),T_x(H_{2}) \rangle$.
\end{lemma}
\begin{proof}
Let $z$ be a wrinkle of $x$.
Let $H=X([x,z])$. Since $\<T_x(H),z\>=\<H\>$, we know that $H\subseteq T_x$.  Let $H'$ be a \EQ\  through $x$ which intersects $H$ in a subspace of minimal dimension. Since $H,T_x(H')\subseteq T_x$, this dimension is at least 1. Since $d\leq 9$, there are four cases.

We first observe that $$H\cap T_x(H')\subseteq \<H\>\cap\<H'|>\subseteq H\cap H'\subseteq H\cap T_x(H'),$$ hence $H\cap T_x(H')=H\cap H'$. 

\textbf{Case 1:  $H$ intersects $H'$ in a line $L$.} \\ In this case $\<H,T_x(H')\>=T_x$ by a simple dimension argument, which also shows $\dim(T_x)=2d$. Since $L$ is a singular line, we see that $\dim\<T_x(H),T_x(H')\>=2d-1$. Hence we can pick a point $u\in (X\cap T_x)\setminus \<T_x(H),T_x(H')\>$, with $\<x,u\>$ singular. Let $H''$ be a symp through $x$ and $u$. Since $\dim\<u,H\>=d+2$, we see that $\dim(\<u,H\>\cap T_x(H'))=2$. Consequently, there is a point $v\in (\<u,H\>\cap\<T_x(H')\>)\setminus \<H\>$. The line $uv$ inside the $(d+2)$-space $\<u,H\>$ intersects the $(d+1)$-space $\<H\>$ in a point $y'$. If $y'$ were contained in $T_x(H)$, then the line $uv=vy'$  would be contained in $\<T_x(H),T_x(H')\>$, contradicting the choice of $u$. Hence $y'\in \<H\>\setminus T_x(H)$. Then there is a unique point $y\in H\setminus T_x(H)$ on the line $xy'$ (and $y$ is a wrinkle of $x$). By replacing $u$ with an appropriate point on $xu$, we may assume that $y'=y$. Hence $y\in uv\subseteq\<T_x(H''),T_x(H')\>$ and the lemma is proved.

\textbf{Case 2: $H$ intersects $H'$ in a plane $\pi$.}  \\
In this case $\dim\<H,T_x(H')\>=2d-1$. If $\dim(T_x)=2d-1$, then the lemma follows from an argument completely similar to the one of Case 1. So we may assume $\dim(T_x)=2d$.  Note that $\dim\<T_x(H),T_x(H')\>=2d-2$. Hence there exist two points $v_1,v_2\in X$ such that $xv_1$ and $xv_2$ are singular lines, $v_1\notin \<T_x(H),T_x(H')\>$, and $v_2\notin\<v_1,T_x(H),T_x(H')\>$. This implies that $v_1v_2$ does not intersect $\<T_x(H),T_x(H')\>$.

If $v_1v_2$ intersects $\<H\>$ in a point $w$ then Axiom~(MM2) implies $w \in X$. Hence $\<v_1,v_2,x\>$ is a singular plane. Since $v_1v_2 \cap T_{x}(H)=\emptyset$, $w \in H\setminus T_x(H)$, contradicting the fact that $xw$ is a singular line.

Hence $\dim(\<v_1,v_2,H\>)=d+3$, and so, putting $U=\<v_1,v_2,H\>$ we have $\dim (U \cap T_x(H'))\geq 3$. Consequently we find a point $u$ in $T_x(H')\setminus\<H\>$, with $u\in U$. In $U$, the plane $\<u,v_1,v_2\>$ intersects the space $\<H\>$ in some point $y$. If $y\in T_x(H)$, then the point $uy\cap v_1v_2$ belongs to $\<T_x(H),T_x(H')\>$, contradicting the choices of $v_1,v_2$. Hence $y\in \<H\>\setminus T_x(H)$. As in Case 1, we can rechoose $u$ on $xu$ so that $y\in H$, and so $y$ is a wrinkle of $x$. It follows that $y\in\<T_x(H'),T_x(H'')\>$, concluding Case 2.

\textbf{Case 3: $H$ intersects $H'$ in a $3$-space $\Sigma$.} \\
Here again, the case where $\dim(T_x)=2d-2$ is completely similar to Case 1, and if $\dim(T_x)=2d-1$, then we can copy the proof of Case 2, adjusting some dimensions.

So assume $T_x$ is $2d$-dimensional. Since  $\dim\<T_x(H),T_x(H')\>=2d-3$, we can find points $v_1,v_2,v_3\in X$ such that $xv_i$ is a singular line, $i=1,2,3$, and  $\<v_1,v_2,v_3,T_x(H),T_x(H')\>=T_x$. By Lemma~\ref{lemma1} and Lemma~\ref{lemma2}, there is a \EQ\ $H_{12}$ containing $x,v_1,v_2$ (note $d\geq 6$ by the definition of $\Sigma$). Since $v_1v_2$, which is contained in $T_x(H_{12})$, does not meet  $\<T_x(H),T_x(H')\>$, the latter intersects $T_x(H_{12})$ in a subspace of dimension at most $d-2$. Lemma~\ref{n-2} implies that we can find a singular line $L$ in $T_x(H_{12})$ skew to  $\<T_x(H),T_x(H')\>$. We may now assume that
$L=v_1v_2$, and possibly rechoose $v_3$ so that we still have $\<v_1,v_2,v_3,T_x(H),T_x(H')\>=T_x$. Note that every point $v$ of the plane $\pi=\<v_1,v_2,v_3\>$ is contained in a quadratic space $\xi$, since $v_3$ together with any point of $v_1v_2$ is contained in one. Hence, as in Case 2 above, $\pi$ does not meet $\<H\>$. Note also that $v$ in fact belongs to $T_x(\xi)$.

So, the space $U=\<\pi,H\>$ is $(d+4)$-dimensional and intersects $T_x(H')$ in a subspace of dimension at least $4$. Similarly as above in Case 2, we find a point $u\in U$ in $T_x(H')\setminus\<H\>$, and the $3$-space $\<u,\pi\>$ intersects $\<H\>$ in some point $y$, which does not belong to $T_x(H)$, and which we may assume to belong to $X$ (by rechoosing $u$ on $xu$). As we noted above, the point $uy\cap\pi$ is contained in $T_x(\xi)$, for some quadratic space $\xi$, and so $y\in \<T_x(\xi),T_x(H')\>$, completing Case 3.

\textbf{Case 4: $H$ intersects $H'$ in a $4$-space $\alpha$.}\\
As before, we may again assume that $\dim(T_x)=2d$, as otherwise the proofs are similar to the previous cases. Note that $\dim\<T_x(H),T_x(H')\>=2d-4$.

Now, similarly to Case 3, we first find points $v_1,v_2\in X$ such that $\<x,v_1,v_2\>$ is a singular plane and $\dim\<v_1,v_2,T_x(H),T_x(H')\>=2d-2$, and then we find, by the same token, points $v_3,v_4\in X$ such that $\<x,v_3,v_4\>$ is a singular plane and $\<v_1,v_2,v_3,v_4,T_x(H),T_x(H')\>=T_x$. Also as in Case 3, every point of the $3$-space $\Sigma=\<v_1,v_2,v_3,v_4\>$ is contained in a quadratic space. But then we can copy the rest of the proof of Case 3 to complete Case 4.

This completes the proof of the lemma.
\end{proof}

Combining Lemma~\ref{generation} with Lemma~\ref{badpoints}, we obtain the following consequence.

\begin{lemma}\label{E6completed}
If $2\leq d\leq 9$, then every point $x\in X$ is smooth.
\end{lemma}

This implies the following inductive result.

\begin{cor}\label{MMinductive}
If $6\leq d\leq 9$, for all $p\in X$, the residue $(X_{p},\Xi_{p})$ is a Mazzocca-Melone set of split type $d-2$.
\end{cor}

\begin{proof}
In view of Lemma~\ref{induction}, is suffices to show that $(X_{p},\Xi_{p})$ satisfies Axiom~(MM3).
Let $x\in X$ be collinear with $p$. By Lemma~\ref{proper} there exists a \EQ\ $H$ through $p$ not containing $x$. Put $W_x=T_x\cap\<H\>$. Since the points of $H$ collinear to $x$ are contained in a maximal singular subspace of $H$, Lemmas~~\ref{quadric},~\ref{paraquad} and~\ref{E6completed}, imply that $\dim(W_x)\leq \lfloor \frac{d+1}{2}\rfloor$.   Hence we can select a subspace $W'_x$ in $T_p(H)$ complementary to $W_x$ (which is indeed contained in $T_p(H)$). Since $\dim(W'_x)\geq\lceil \frac{d-1}{2} \rceil-1$, we see that $\dim(T_x\cap T_p)\leq \lfloor \frac{3d+1}{2} \rfloor$. Hence, if $x'\in X_p$ corresponds to $x$, then $T_{x'}$ has dimension at most $\lfloor \frac{3d+1}{2} \rfloor-1$. For $d=6,7,8,9$ this yields $8,10,11,13$, respectively, which finishes the proof.
\end{proof}

The strategy of the proof of Main Result 1 for $d\in\{4,5\}$ will also be to show that the residue satisfies Axiom (MM3). However, this will need some very particular arguments. Concerning the cases $d\geq 10$, the arguments to prove Lemma~\ref{generation} cannot be pushed further to include higher dimensions. Hence we will need  different arguments, which, curiously, will not be applicable to any of the cases $d\leq 9$; see Proposition~\ref{casengeq5}.

\begin{rem}\rm\label{secant}
The varieties in the conclusion of the Main Corollary are analogues of complex Severi varieties, as already mentioned. The varieties in the conclusion of the Main Result that are not in the conclusion of the Main Corollary have, in the complex case, as secant variety the whole projective space. It can be shown that this remains true over an arbitrary field; this is rather easy for the Segre and line Grassmannian varieties, and slightly more involved for the half-spin variety $\mathcal{D}_{5,5}$, although the calculations are elementary, but tedious, using \cite{Wells}. Likewise, the secant variety of the secant variety of a complex Severi variety is the whole projective space, and this remains true for our analogues over an arbitrary field. This now has the following consequence.
\end{rem}

\begin{prop}\label{universal}
Let $X$ be one of the varieties in the Main Result for $d\geq 2$, and regard $X$ as a set of points. Let $\Gamma=(X,\cL)$ be the point-line geometry obtained from $X$ by collecting (in the set $\cL$) all projective lines entirely contained in $X$, and considering the natural incidence relation. Let $\Xi'$ be the family of all maximal split quadrics contained in $X$. Then there is a unique projective embedding of $\Gamma$ with the property that any two projective subspaces generated by members of $\Xi'$ intersect in a subspace all points of which belong to $X$.
\end{prop}

\begin{proof}
Every geometry $\Gamma$ has a natural embedding as given in Section~\ref{examples}. This embedding is always the absolute universal one, meaning that every other embedding is a projection of the absolute universal one. This follows by results of Zanella ($d=2$), see Theorem~3 of \cite{Zan}, Havlicek ($d=4$), see \cite{Hav}, Wells ($d=6$), see  Theorem~5 of \cite{Wells}, and combining Theorem 6.1 of Cooperstein \& Shult \cite{Coo-Shu} (which bounds the dimension of any embedding of $\cE_{6,1}$ to $26$) and Paragraph 4.11.1 of Kasikova \& Shult \cite{KS} (which shows the existence of the absolute universal embedding for $d=8$).  Now either the secant variety is the whole projective space (in which case there are no proper projections), or the secant variety of the secant variety is the whole projective space (in which case every projection reduces the dimension of the space spanned by some pair of maximal split quadrics, contradicting Axiom~(MM2)). This shows the assertion. 
\end{proof}

In view of the previous proposition, is suffices to show that the points and singular lines of a Mazzocca-Melone set of split type constitute the geometries related to one of the varieties in the Main Result, and only those. 
The geometries $\Gamma$ defined above related to a line Grassmannian variety, half-spin variety or variety $\cE_{6,1}(\K)$ will be referred to as a \emph{line  Grassmannian geometry, half-spin geometry} or \emph{geometry $\cE_{6,1}(\K)$}, respectively. These geometries are related to (Tits-)buildings by defining additional objects in the geometry; it is in this way that we will recognize the varieties of our Main Result. This is the content of the next section.

\section{Recognizing the line Grassmannian, half-spin and $\cE_{6,1}$ geometries}\label{exist}

We start by recognizing the line Grassmannian geometries. Let  $(X,\Xi)$ be a proper Mazzocca-Melone set of split type $4$ in $\PAG^{N}(\K)$. For each point $p\in X$, we denote the residue at $p$ by $(X_p,\Xi_p$, and we also denote the subspace spanned by $X_p$ by $C_p$, as before (see Lemma~\ref{induction}). We will first recognize these residues, and then the whole geometry. To that aim, we need one more lemma.

\begin{lemma}\label{Nofour}
Let $p\in X$. If $k\geq 4$, then there are no singular $k$-spaces in the residue $(X_p,\Xi_p)$.
\end{lemma}
\begin{proof}
Suppose otherwise and let $M$ be a maximal singular $k$-space in $(X_p,\Xi_p)$, with $k\geq 4$. Note that, by repeated use of Lemma~\ref{lemma1}, since, by Lemma~\ref{induction}, $(X_p,\Xi_p)$ is a proper pre-Mazzocca-Melone set of split type $2$, no point of $X_p$ outside $M$ is collinear with at least two points of $M$. Hence every singular line meeting $M$ in a point $x$ lies together with any point of $M$ distinct from $x$ in a unique \EQ. We will use this observation freely.

First we claim that  every point $r$ of $M$ is contained in at least two singular lines not contained in $M$. Indeed, by considering any symp through $r$, we see that $r$ is contained in at least one such singular line, say $L_r$. Suppose now that some point $s\in M$ is contained in at least two lines $L_s,L_s'$ not contained in $M$. Then the symps defined by $\<r,s\>$ and $L_s$, and by $\<r,s\>$ and $L_s'$ both contain $\<r,s\>$, so they cannot both contain $L_r$. Consequently there exists a second singular line through $r$ not contained in $M$. Hence we may assume that through every point $r$ of $M$ passes a unique singular line $L_r$ not contained in $M$. 

Consider three points $x,y$ and $z$ spanning a plane $\pi$ in $M$. Consider a point $x'$ on $L_{x}$ different from $x$. Using a symp through $L_x$ and $L_y$, we deduce that  there is a unique point $y'$ on $L_{y}$ collinear with $x'$, and similarly, there is a unique point $z'$ on $L_{z}$ collinear with $x'$. If $x',y'$ and $z'$ were contained in a common line $L$, then the \EQ s $X([x,y'])$ and $X([y',z])$ would share the lines $L$ and $L_y$, and hence coincide, a contradiction since  this symp then contains the plane $\pi$. If $y'$ and $z'$ are not collinear then, by the Quadrangle Lemma, the \EQ\ $X([y',z'])$ contains $x'$, $\<y,y'\>$ and $\<z,z'\>$, and hence it contains three singular lines through $y'$, a contradiction. Hence the plane $\pi':=\langle x',y',z' \rangle$ is singular. Consider next a 3-space $\Pi$ in $M$ containing $\pi$. This leads with similar reasonings to a singular $3$-space $\Pi'$ containing $\pi'$. But then, since $\dim(T_p)\leq 8$, and so $\dim(C_p)\leq 7$,   $\Pi'\cap M$ is non-empty, contradicting our assumption of the previous paragraph. This completes the proof of our claim.

Hence there exist lines $L_{x}$ and $L_{y}$ through $x$ and $y$, respectively, such that the \EQ s $H_{x}$ and $H_y$ containing $L_{x},y$ and $L_y,x$, respectively, are distinct.

As $\dim T_{p}\leq 8$, the $(k+2)$-spaces $\langle H_{x},M\rangle$ and $\langle H_{y},M\rangle$ intersect at least in a $(2(k+2)-7)$-space. Since $2k-3\geq k+1$ (because $k\geq 4$), those spaces share a $(k+1)$-space $S$ containing $M$. Now $S$ intersects $\<H_x\>$ and $\<H_y\>$ in planes containing $\<x,y\>$ and so $S$ contains singular lines $R_x$ and $R_y$ of $H_x$ and $H_y$ meeting $\<x,y\>$ in the single points $p_x$ and $p_y$, respectively. Now
pick two points $r_{x}$ and $r_{y}$ on $R_{x}$ and $R_{y}$, respectively, and not contained in $M$. Then $R=\langle r_{x},r_{y} \rangle\subset S$ intersects $M$ in a point, which necessarily belongs to $X_p$ since $M\subseteq X_p$, and hence is a singular line (this also implies $p_x\neq p_y$). The Quadrangle Lemma applied to $xy,R_{x},R,R_{y}$ now implies that $H_{x}=H_{y}$, a contradiction. The lemma is proved.
\end{proof}

\begin{prop}\label{TSP}
Exactly one of the following holds
\item[$(1)$] For all $p$, $(X_{p},\Xi_{p})$ is isomorphic to the Segre variety $\mathcal{S}_{1,3}(\K)$.
\item[$(2)$] For all $p$, $(X_{p},\Xi_{p})$ is isomorphic to the Segre variety $\mathcal{S}_{1,2}(\K)$.
\end{prop}
\begin{proof}
First we prove that $(X_p,\Xi_{p})$ is a Mazzocca-Melone set of split type~$2$ in the space $C_p$ of dimension at most 7. By Lemma~\ref{induction}, we only have to show that Axiom (MM3) holds.

Suppose that (MM3) is not satisfied. Then there exist a point $x$ contained in a \EQ\ $H$ and  three singular lines $L_{1},L_{2}$ and $L_{3}$ containing $x$ and spanning a 3-space $\Pi$ which intersects $\langle H \rangle$ only in $x$. We distinguish two cases.

\textbf{Case 1.} Suppose that at least two of the three planes determined by $L_{1},L_{2}$ and $L_{3}$, say $\langle L_{1},L_{2} \rangle $ and $\langle L_{2},L_{3} \rangle$, are not singular. Then the \EQ s $H_{{1},{2}}$ (through $L_1$ and $L_2$) and $H_{{2},{3}}$ (through $L_2$ and $L_3$) intersect $H$ only in $x$. Consequently the subspaces $\<H,H_{1,2}\>$ and $\<H,H_{2,3}\>$ are $6$-dimensional and hence have a $5$-space $\Sigma$ containing $H$ and $L_2$ in common, which meets each of $\<H_{1,2}\>$ and $\<H_{2,3}\>$ in a plane $\pi_1$ and $\pi_3$, respectively. These planes contain $L_2$ and hence contain further singular lines of $H_{1,2}$ and $H_{2,3}$, say $R_1$ and $R_3$, respectively.

By assumption we have $(R_1,R_3)\neq(L_1,L_3)$ (the latter pair is not contained in a $5$-space together with $H$ and $L_2$). Suppose now that $\<R_1,R_3\>$ is a $3$-space. Then the latter meets $\<H\>$ in a line $R$, every point of which is on a (necessarily singular) line meeting $R_1$ and $R_3$. Let $r$ be any point of $R$. Then, since $R_1,R_3$ are two non-intersecting lines and $r$ belongs to the 3-space spanned by $R_1$ and $R_3$, there is a line $T_r$ containing $r$ and intersecting $R_1$ and $R_3$ in two points, say $r_1$ and $r_3$, respectively. Any symp through $r_1$ and $r_3$ must meet de symp $H$ in at least the point $r$ (by (MM2)) and so $T_r$ is a singular line. Now. we can choose $r$ such that $T_r$ does not intersect $L_2$, and then we have the quadrangle $T_r,R_1,L_2,R_3$. The Quadrangle Lemma implies that there is a unique (because the lines are distinct) symp $H^*$ containing this quadrangle. But $H^*$ contains $R_1$ and $L_2$, hence must coincide with $H_{1,2}$. Similarly, $H^*=H_{2,3}$. Consequently, $H_{1,2}=H_{2,3}$, a contradiction.  

So we may assume that $R_1$ and $R_3$ meet in a point $y$ of $L_2$ not belonging to $H$. Then the plane $\<R_1,R_3\>$ meets $\<H\>$ in a point $z$, and this point is easily seen to belong to $H\setminus L_2$.  If $\<x,z\>$ is a singular line, then we have the singular planes $\<R_1,R_3\>$ and $\<L_2,z\>$, which have the singular line $\<y,z\>$ in common. Lemma~\ref{lemma1} implies that $\<R_1,L_2,R_3\>$ is a singular $3$-space containing the non-singular plane $\<R_1,L_2\>\subseteq\<H_{1,2}\>$, a contradiction. Suppose now that the line  $\<x,z\>$ is not singular. Since $\<R_1,R_3\>$ is a singular plane, the line $\<y,z\>$ is singular. So we can apply the Quadrangle Lemma to the lines $L_2$, $\<y,z\>$ and two lines of $H$ connecting $z$ with $x$, obtaining that there is a symp through $x,z$ and $L_2$. Hence  $H$ contains $L_2$, a contradiction.


\textbf{Case 2.} Suppose that both $\langle L_{1},L_{2} \rangle$ and $\langle L_{1},L_{3} \rangle$ are singular planes.
Then $\Pi$ is a singular $3$-space by Lemma \ref{lemma1}. Let $N_1$ and $N_2$ be the singular lines of $H$ meeting $L_1$. If any of the planes $\langle L_{i},N_{j} \rangle$, $i\in\{1,2,3\}$, $j\in\{1,2\}$, is singular then by Lemma \ref{lemma1} we obtain the singular 4-space $\langle \Pi,N_{j} \rangle$, a contradiction to Lemma~\ref{Nofour}. Hence we jump back to case (1) using the symp defined by $L_{1},N_{1}$ instead of $H$, and the lines $L_{2},N_{2},L_{3}$ instead of $L_1,L_2,L_3$, respectively.

Hence (MM3) is satisfied and $(X_{p},\Sigma_{p})$ is a Mazzocca-Melone set of split type $2$. Since $\dim C_{p}\leq 7$, we know by \cite{JSHVM3} that $X_{p}$ is either $\mathcal{S}_{1,2}(\K)$ or $\mathcal{S}_{1,3}(\K)$. Suppose $(X_p,\Xi_p)$ is not always isomorphic to $\mathcal{S}_{1,2}$. Then there exists a point $r$ such that $X_{r}$ is $\mathcal{S}_{1,3}(\K)$. For every point $q$ collinear with $r$, the residue $X_{q}$ is also isomorphic to $\mathcal{S}_{1,3}(\K)$ as there is a 4-space containing $q$ and $r$. Now consider a point $s$ not collinear with $r$. Then in $X([r,s])$ there are points collinear with both $r$ and $s$. Repeating the above argument twice yields that $X_{s}$ is projectively equivalent to $\mathcal{S}_{1,3}(\K)$.
\end{proof}

We now define an incidence geometry $\mathcal{G}(X)$ with elements of Type $1$, $2$ and $3$, which we will prove to be a projective space of dimension $5$ or $4$, corresponding to the conclusion (1) or (2), respectively, of Proposition \ref{TSP}. We will treat both cases at once, so let $f$ be in $\{3,4\}$.
The elements of Type 1 are the singular $f$-spaces of $X$, the elements of Type 2 are the points of $X$, and the elements of Type 3 are the singular planes of $X$ not contained in  $f$-spaces. The incidence relation is inclusion made symmetric except for elements of Type 1 and $3$, which we declare incident if the corresponding singular $f$-space intersects the corresponding singular plane in a line.

We now show that the Type 1, 2 and 3 elements of $\mathcal{G}(X)$, with the given incidence, are the points, lines and planes of a projective space over $\K$ (but not of dimension 3). Notice that collinear points of $X$ correspond to ``concurrent'' elements of Type 2 in $\mathcal{G}(X)$.

\begin{lemma}
Every pair of distinct elements of $\cG(X)$ of Type $1$ is incident with a unique element of Type $2$.
\end{lemma}

\begin{proof}
We have to show that two singular $f$-spaces intersect in a unique point of $X$. First note that two such $f$-spaces intersect in at most a point, since, if they would have a line in common, we can consider the residue in a point of that line and obtain a contradiction in view of Proposition~\ref{TSP} (two $(f-1)$-spaces always belong to the same family of non-intersecting subspaces of the corresponding Segre variety).

Now consider two singular $f$-spaces $\Pi_{1}$ and $\Pi_{2}$ and let $p_{1}\in \Pi_{1}$ and $p_{2}\in \Pi_{2}$ be two arbitrary points. Consider a symp $H$ containing $p_1,p_2$. Then $H$ intersects both $\Pi_{1}$ and $\Pi_{2}$ in planes $\pi_1,\pi_2$, respectively (this follows from Proposition~\ref{TSP} by considering the residues in $p_1$ and $p_2$ and noting that in the Segre variety $\mathcal{S}_{1,f-1}(\K)$, viewed as Mazzocca-Melone set, every symp intersects every singular $(f-1)$-space in a line). We claim that $\pi_1$ intersects $\pi_2$ in a point.

Indeed, if not, then they are disjoint and we can find a singular plane $\pi$ intersecting $\pi_1$ in a line $L$ and $\pi_2$ in a point $x$. Looking in the residue of a point on $L$, we conclude with the aid of  Proposition~\ref{TSP} that $\pi$ is a   maximal singular subspace. But looking in the residue of $x$, we see that $\pi$ corresponds to a line $L$ in a Segre variety $\mathcal{S}_{1,f-1}(\K)$ disjoint from the $(f-1)$-dimensional subspace corresponding to $\Pi_2$ (maximal singular subspaces of different dimension in a Segre variety always interest in a unique point). Hence $L$ itself is contained in an $(f-1)$-dimensional subspace of $\mathcal{S}_{1,3}(\K)$. Consequently $\pi$ is contained in a singular $f$-space, a contradiction. The claim and the lemma are proved.
\end{proof}

\begin{lemma}
Every pair of distinct elements of Type $2$ of $\mathcal{G}(X)$ which are incident with a common element of Type $1$, are incident with a unique element of type $3$.
\end{lemma}

\begin{proof}
Translated back to $(X,\Xi)$, it suffices to prove that, if $p_1,p_2$ are two points of a singular $k$-space $U$, then the line $L:=\<p_1,p_2\>$  is contained in a unique singular plane $\pi$ which is a maximal singular subspace. Existence follows from considering a symp through $\<p_1,p_2\>$; uniqueness follows from Lemma~\ref{lemma1} (one could also use Proposition~\ref{TSP} here).
\end{proof}

\begin{lemma}
The elements of Type $1$ and $2$ incident with a given element of Type $3$ induce the structure of a projective plane over $\K$.
\end{lemma}

\begin{proof}
This follows immediately from the definition of $\mathcal{G}(X)$, and the fact that every singular line is contained in a unique singular $f$-space, by Proposition~\ref{TSP}.
\end{proof}

\bigskip

It follows from Theorem 2.3 of \cite{BC} that $\mathcal{G}(X)$ is a projective space. We determine the dimension.

\begin{lemma}
The dimension of $\mathcal{G}(X)$, as a projective space, equals $f+1$.
\end{lemma}

\begin{proof}
Fix an element of Type $1$, which is a singular $f$-space $\Pi$. The elements of Type~2 incident with it are the points $x\in\Pi$, and the elements of Type~3 incident with $\Pi$ are the singular planes not contained in $\Pi$ intersecting $\Pi$ in a line. Since there is only one such singular plane for a given line of $\Pi$, we may identify each such plane with that line. Incidence between the Type~2 and Type~3 elements is the natural incidence in $\Pi$, and so we see that Type~2 and~3 elements incident with $\Pi$ form a projective space of dimension $f$. Hence the dimension of $\cG(X)$ as a projective space is equal to $1+f$.   
\end{proof}

We can now prove the case $d=4$ of the Main Result.

\begin{prop}\label{casen=2}
A Mazzocca-Melone set of split type $4$ in $\PAG^{N}(\K)$  is projectively equivalent to one of the following:
\begin{itemize*}
\item the line Grassmannian variety $\mathcal{G}_{3,1}(\K)$, and then $N=5$;
\item the line Grassmannian variety $\mathcal{G}_{4,1}{(\K)}$, and then $N=9$;
\item the line Grassmannian variety $\mathcal{G}_{5,1}{(\K)}$, and then $N=14$;
\end{itemize*}
The first case is not proper and corresponds to a hyperbolic quadric of Witt index $3$.
\end{prop}

\begin{proof}
By the previous lemmas, if $(X,\Xi)$ is proper, then $\cG(X)$ is a projective space of dimension $4$ or $5$ over $\K$. Now, the lines of $\cG(X)$ are the points of $X$, and the planar line pencils are the points in the non-trivial intersection of a singular $4$-space and a singular $2$-space not contained in a singular $4$-space (this follows from Proposition~\ref{TSP} by picking a point in the intersection).  

Consequently $X$, endowed with the singular lines, is an embedding of the line Grassmannian geometry $\cG_{4,1}(\K)$ or $\cG_{5,1}(\K)$. Proposition~\ref{universal} implies the uniqueness of this embedding, and the assertions follow. 
\end{proof}

\bigskip

Now we assume $d\in\{6,8\}$. 

In order to identify the half-spin geometry $\cD_{5,5}(\K)$ and the geometry $\cE_{6,1}(\K)$, we will use some basic theory of  \emph{diagrams}, as developed by Tits, culminating in the beautiful characterizations of building geometries by their diagrams in \cite{Tits83}. This works as follows.

A building is in fact a numbered simplicial chamber complex satisfying some axioms. One of the axioms says that every pair of chambers (i.e., every pair of simplices containing a vertex of every type) is contained in a thin subcomplex (thin means that every panel|a panel is a simplex obtained from a chamber by removing one vertex|is contained in exactly two chambers). The other axioms then ensure that all these thin subcomplexes (called apartments) are isomorphic and when they are finite, then we talk about a spherical building. The type preserving automorphism group of each apartment is a Coxeter group. The type of the Coxeter group is also the type of the building. Attached to each Coxeter group is a Coxeter diagram. In most spherical types, this Coxeter diagram is just a Dynkin diagram with the arrow removed. In our current case, types $\mathsf{D_5}$ and $\mathsf{E_6}$, we have the diagrams

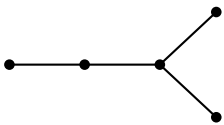
\begin{center}
\begin{tikzpicture}[style=thick]
\foreach \x in {0,1,2}{
\fill (\x,0) circle (2pt);}

\fill (2.75, .7) circle (2pt);
\fill (2.75, -.7) circle (2pt);
\draw (2.75, .7) -- (2,0) -- (2.75, -.7); 
\draw (0,0) -- (2,0); 
\end{tikzpicture}
\end{center}

and 

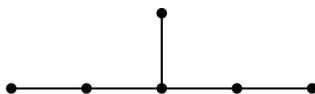
\begin{center}
\begin{tikzpicture}[style=thick]
\foreach \x in {0,1,2,3,4}{
\fill (\x,0) circle (2pt);}
\fill (2,1) circle (2pt);
\draw (0,0) -- (4,0); 
\draw (2,1) -- (2,0);
\end{tikzpicture}
\end{center}

respectively. This diagram now acts as an identity card for the geometry of the building. Each node stands for a type of vertices, and each (simple) edge stands for a family of residues isomorphic to the flag complex of a projective plane. Namely, let the nodes of an edge be labeled $i$ and $j$ (so these are types of vertices of the building); then, for any simplex $F$ obtained from a chamber by removing the vertices of types $i$ and $j$, the graph formed by the vertices of types $i$ and $j$ completing $F$ to a panel, two vertices being adjacent if they form a simplex (this graph is a residue of type $\{i,j\}$), is the incidence graph of a projective plane. If the nodes labeled $i$ and $j$ are not on an edge in the diagram, then the similarly defined graph is complete bipartite. Likewise, one defines the residue of any simplex. The rank of a residue is the difference between the size of a chamber and the size of the simplex in question. 

Now, given a numbered chamber complex $\cC$, we can define residues as in the previous paragraph. Given a graph $\Gamma$, we say that $\cC$ conforms to $\Gamma$ if the nodes of $\Gamma$ can be identified with the types of vertices in such a way that, for each pair of types $i.j$, a residue of type $\{i,j\}$ is the incidence graph of a projective plane if $\{i,j\}$ is an edge of $\Gamma$, and a complete bipartite graph otherwise.   

Now it follows from the main result in \cite{Tits83} and from work of Brouwer \& Cohen \cite{BrCo} that a numbered chamber complex is isomorphic to a building of type $\mathsf{D}_5$ or $\mathsf{E_6}$ if and only if it conforms to the Coxeter diagram of type $\mathsf{D}_5$ or $\mathsf{E}_6$, respectively, and if all residues of rank at least 2 are connected, i.e., in each residue of rank at least 2, the graph on the vertices where two vertices are adjacent if they form a simplex, is connected (we say that the complex is residually connected). This means that, in order to prove that the Mazzocca-Melone sets in question are really the half-spin geometry $\cD_{5,5}$ and the $\cE_{6,1}$ variety, respectively, it suffices to define a residually connected numbered simplicial chamber complex where the points of he Mazzocca-Melone set are the vertices of type $5$ and $1$, respectively, and the singular lines are the vertices of type $3$ (we use standard Bourbaki labelling, see \cite{Bourbaki}), and to show that it conforms to the diagram $\mathsf{D_5}$ and $'mathsf{E_6}$, respectively. That is exactly how we are going to proceed.
 
\begin{prop}\label{casen=3}
A Mazzocca-Melone set $(X,\Xi)$ of split type $6$ in $\PAG^{N}(\K)$  is projectively equivalent to one of the following:
\begin{itemize*}
\item the half-spin variety $\mathcal{D}_{4,4}(\K)$, and then $N=7$;
\item the half-spin variety $\mathcal{D}_{5,5}(\K)$, and then $N=15$.
\end{itemize*}
The first case is not proper and corresponds to a hyperbolic quadric of Witt index $4$.
\end{prop}

\begin{proof}
We may assume that $(X,\Xi)$ is proper, as otherwise we have the half-spin variety $\mathcal{D}_{4,4}(\K)$, which, by triality, is isomorphic to $\mathcal{D}_{4,1}$. By Lemma~\ref{MMinductive} we know that every residue is a Mazzocca-Melone set of split type $4$, which is proper by Lemma~\ref{proper}. Since the ambient space of the residue has dimension at most 12 by (MM3), Proposition~\ref{casen=2} implies that the residue corresponds to a line Grassmannian variety $\mathcal{G}_{4,1}(\K)$.

Define the following numbered simplical complex $\cG(X)$ with vertices of types 1 up to 5. The vertices of type 1 are the symps of $(X,\Xi)$, the ones of type 2 are the singular 3-spaces not contained in a singular 4-space, the type 3 vertices are the singular lines, the type 4 ones are the singular 4-spaces and, finally, the vertices of type 5 are the points belonging to $X$. We define an adjacency relation between two vertices of distinct type as containment made symmetric, except that a vertex of type 1 is incident with a vertex of type 4 if the corresponding 4-space intersects the corresponding hyperbolic quadric in a 3-space, and also that a vertex of type 2 is incident with one of type 4 if the corresponding 3-space intersects the corresponding 4-space in a plane. The simplices of $\cG(X)$ are the cliques of the corresponding graph. 

It is straightforward to verify that this simplical complex is a chamber complex. We show that it conforms to the diagram of type $\mathsf{D}_{5}$ (where we have chosen the types above so that they conform to the Bourbaki labeling \cite{Bourbaki}), which, by Tits \cite{Tits83}, implies that $\cG(X)$ is the building of type $\mathsf{D}_5$ over $\K$. It follows easily from the definition of incidence in $\cG(X)$ that any residue of type $\{1,2,3,4\}$ of $\cG(X)$ (say, of an element $p$ of type $5$) corresponds to the residue $(X_p,\Xi_p)$ of $(X,\Xi)$ as defined in the present paper. Hence all rank 2 residues of type $\{i,j\}$, with $\{i,j\}\subseteq\{1,2,3,4\}$, are correct. Also, in the same way, the residues of type $\{2,3,4,5\}$ of $\cG(X)$ correspond to the geometry of the symps, which establishes the correctness of all rank 2 residues of type $\{i,j\}$, for all $\{i,j\}\subseteq\{2,3,4,5\}$. It remains to check the residues of type $\{1,5\}$. But these all are trivially complete bipartite graphs.

With the information of the previous paragraph, it is not difficult to see that $\cG(X)$ is residually connected.
Since type 5 vertices of $\cG(X)$ correspond to elements of $X$, and type $3$ to the singular lines, $X$ is an embedding of the half-spin geometry $\mathsf{D}_{5,5}(\K)$. Proposition~\ref{universal} completes the proof of the proposition.  
\end{proof}

\begin{prop}\label{casen=4}
A Mazzocca-Melone set $(X,\Xi)$ of split type $8$ in $\PAG^{N}(\K)$  is projectively equivalent to one of the following:
\begin{itemize*}
\item the variety $\mathcal{D}_{5,1}(\K)$, and then $N=9$;
\item the variety $\mathcal{E}_{6,1}(\K)$, and then $N=26$.
\end{itemize*}
The first case is not proper and corresponds to a hyperbolic quadric of Witt index $5$.
\end{prop}

\begin{proof}
We may assume that $(X,\Xi)$ is proper. By Lemma~\ref{MMinductive} we know that every residue is a Mazzocca-Melone set of split type $6$, which is proper by Lemma~\ref{proper}.
Hence it follows from Proposition~\ref{casen=3} that every residue is a half-spin variety $\mathcal{D}_{5,5}(\K)$.

Define the following numbered simplical complex $\mathcal{G}(X)$ with vertices of types 1 up to 6.
The \emph{vertices of type} $1,\ldots, 6$ are the {points of $X$, singular 5-spaces, singular lines, singular planes, singular 4-spaces which are not contained in a singular $5$-space}, and {\EQ}s, respectively.
 First defining a graph on these vertices, adjacency is containment made symmetric, except in the following two cases. A singular 5-space and a singular 4-space which is not contained in a singular $5$-space are adjacent if they intersect in a 3-space, and a singular {5-space} and a \EQ\ are adjacent if their intersection is a 4-dimensional singular space. The simplices of $\cG(X)$ are again the 
cliques of the graph just defined. It is routine to check that this is a chamber complex. 

We claim that $\mathcal{G}(X)$ is a building of type $\mathsf{E}_6$ over the field $\K$, with standard Bourbaki \cite{Bourbaki} labeling of the types.

We now show that $\cG(X)$ conforms to the diagram of type $\mathsf{E}_6$. Since the residue in $\cG(X)$ of an element of type $6$ is a \EQ, we already know that the residues of types $\{i,j\}\subseteq\{1,2,3,4,5\}$ are the correct ones. Moreover, it follows directly from the definition of $\mathcal{G}(X)$ that the residue of an element $p$ of type 1 in $\mathcal{G}(X)$ corresponds to the residue $(X_p,\Xi_p)$ in $(X,\Xi)$. Since the latter is the half-spin variety $\cD_{5,5}(\K)$, the residues of type $\{i,j\}\subseteq\{2,3,4,5,6\}$ are also the correct ones. It remains to check that the residues of type $\{1,6\}$ are complete bipartite graphs, which is straightforward.

With the information of the previous paragraph, it is not difficult to see that $\cG(X)$ is residually connected.
Since type 1 elements of $\cG(X)$ correspond to elements of $X$, and type $3$ to the singular lines, $X$ is an embedding of the geometry $\mathsf{E}_{6,1}(\K)$. Proposition~\ref{universal} completes the proof of the proposition. 
\end{proof}

\section{The cases of non-existence}\label{notexist}
In this section, we show the nonexistence of proper Mazzocca-Melone sets of split type $d$ for $d\in\{5,7,9\}$ and $d\geq 10$. There are essentially two arguments for that: one ruling out $d=5$ (and then use Corollary~\ref{MMinductive} to get rid of $d=7,9$ as well), and one for $d\geq 10$. Interestingly, the argument to rule out the general case $d\geq 10$ can not be used for smaller values of $d$, whereas the argument in Corollary~\ref{MMinductive} cannot be pushed further to include higher values of $d$.  

\subsection{Mazzocca-Melone sets of split type 5, 7 and 9}

By Corollary~\ref{MMinductive}, a proper Mazzocca-Melone set of split type $7$ or $9$ does not exist as soon as we show that no such set of split type $5$ exists. The latter is the content of this subsection.

So let $(X,\Xi)$ be a proper Mazzocca-Melone set of split type $5$. Select a point $p\in X$ and consider the residue $(X_p,\Xi_p)$, which is a pre-Mazzocca-Melone set of split type $3$ in $\PAG^k(\K)$, for $k\leq 9$. We note that the proof of Corollary~\ref{MMinductive} implies that the tangent space at each point of $(X_p,\Xi_p)$ has dimension at most $7$. If the dimension is $6$, for every point, then we have a proper Mazzocca-Melone set and have reached our desired contradiction, in view of  \cite[Proposition 5.4]{JSHVM4}. So we may assume that the dimension is $7$, for at least one point.

Our argument is based on the following observation.

\begin{lemma}\label{crucial5}
Let $x\in X_p$ be such that $\dim T_x=7$, and let $H_1,H_2,H_3$ be three different symps of $(X_p,\Xi_p)$ containing $x$. Suppose $\dim\<H_1,H_2\>=8$, suppose also $\dim(\<H_2,H_3\>)=7$ and $\dim\<T_x(H_1),T_x(H_2),T_x(H_3)\>=7$. Then  $\dim\<H_1,H_3\><8$.
\end{lemma}

\begin{proof}
Suppose, by way of contradiction, that $\dim\<H_1,H_3\>=8$.

Let $\Sigma$ be a subspace of $C_p$ ($=\<X_p\>$) complementary to $\<H_1\>$ and consider the projection of $X_p\setminus H_1$ with center  $\<H_1\>$ onto $\Sigma$.  By assumption, $\<H_2\>$ and $\<H_3\>$ are projected onto $3$-spaces of the $(k-5)$-space $\Sigma$. Suppose, by way of contradiction, that some point $y_2$ of $H_2$ not collinear to $x$ is projected onto the same point as some point $y_3$ of $H_3$ not collinear to $x$. Then the line $\<y_2,y_3\>$ is a singular line intersecting $H_1$ in some point $y_1$. The Quadrangle Lemma immediately implies that $y_1$ is not collinear with $x$. Now, by assumption, $H_2$ and $H_3$ share a line $L$ containing $x$. There is a unique point $z_i$ on $L$ collinear to $y_i$ in $H_i$, for $i\in\{2,3\}$. If $z_2\neq z_3$, then the Quadrangle Lemma applied to $y_2,z_2,z_3,y_3$ leads to $H_2=H_3$, a contradiction. If $z_2=z_3$, then the Quadrangle lemma implies that $\<y_2,y_3,z_2\>$ is a singular plane, which contains $y_1$, implying that $\<y_1,z_2\>$ is a singular line. If $u$ is any point of $H_1$ collinear with both $x$ and $y_1$, then the Quadrangle Lemma applied to $x,u,y_1,z_2$ implies $z_2\in H_1$, so the only possibility is $z_{2}=x$. But then $x$ is collinear with $y_{2}$, a contradiction. We conclude that no point of $H_2$ not collinear to $x$ is projected onto the same point as some point of $H_3$ not collinear to $x$.

If $k\leq 8$, then $k=8$ and since the set of points of $H_i$, $i=2,3$, not collinear to $x$ is projected surjectively onto some affine part of $\Sigma$, these two affine parts intersect nontrivially, contradicting the previous paragraph.

Hence $k=9$. Let $U_i$ be the span of the projection of $H_i$, $i=2,3$. Then $U_2\cap U_3$ is a plane $\pi$, which contains the projection $x_L$ of $L$. Let $\alpha_i$ be the projection of $T_x(H_i)\setminus\{x\}$, $i=2,3$. Then every point of $U_i\setminus \alpha_i$ is the projection of a unique point of $H_i$ not collinear with $x$. The first paragraph of this proof implies that $\pi$ cannot intersect both $\alpha_2$ and $\alpha_3$ in a line. Hence we may assume that $\alpha_2=\pi$. The condition $\dim\<T_x(H_1),T_x(H_2),T_x(H_3)\>=7$ forces $\alpha_3\neq \pi$. So it is easy to see that we can find a point in $\pi$ which is the projection of a point $w\in H_3$ not collinear to $x$, and also the projection of a singular line $K$ of $H_2$ through $x$. In the $5$-space generated by $H_1$ and $K$, the plane $\<w,K\>$  intersects $\<H_1\>$ in a line, which must, by Axiom (MM2), completely belong to $X$. Now it is easy to see that this forces $\<w,K\>$ to be a singular plane, contradicting the fact that $w$ and $x$ are not $X_p$-collinear.
\end{proof}

We also need another lemma in case the field $\K$ is finite.

\begin{lemma}\label{remtruuk}
Let $|\K|=q$ be finite and consider  two non-$X_p$-collinear points $x,y$ in $X_p$. If there are $v$ singular lines through $y$ inside $X_p$, and there are no singular planes through $y$ intersecting $X_p([x,y])$ in a line, then there are at least  $(v-q-1)q+1$ different symps in $X_p$ containing $x$.  
\end{lemma}

\begin{proof}
Exactly $v-q-1$ singular lines through $y$ are not contained in  the symp $\xi$ determined by $x$ and $y$. On these lines lie in total $(v-q-1)q+1$ points (including $y$). The Quadrangle Lemma implies that no two symps defined by these points and $x$ coincide. The lemma follows. 
\end{proof}

We now consider the residue $((X_p)_{p'},(\Xi_p)_{p'})$ of a point $p'$ of $X_p$ and assume that it spans a space of dimension $6$. We denote this residue briefly by $(X',\Xi')\subseteq\PAG^6(\K)$. Its structure is explained by Lemma~\ref{residuesplittype3}, from which we in particular recall that, if $X'$ does not contain singular lines. then $(X',\Xi')$ is a pre-Mazzocca-Melone set of split type 1. 


\begin{lemma}\label{semiaffine}
Let $C$ be a conic of $(X',\Xi')$ and $x\in X'\setminus C$. Then there exists at most one conic containing $x$ and disjoint from $C$.
\end{lemma}

\begin{proof}
Suppose that there are at least two conics $D_1,D_2$ containing $x$ and disjoint from $C$. We claim that $\dim(\<C,D_1,D_2\>)=5$. Indeed, this dimension is at least 5, by Axiom~(MM2) and the fact that $D_1$ and $C$ are disjoint, hence span a $5$-space. So suppose the dimension is 6 (this is the maximum). Denote by $H_1$ the symp in $X_p$  corresponding to $C$, likewise for $H_2$ corresponding to $D_1$ and $H_3$ to $D_2$. Then, since the intersection of $\<H_i\>$ and $\<H_j\>$ is contained in a singular subspace through $p'$, $i,j=1,2,3$, $i\neq j$, we see that 
$\dim \<H_1,H_2\>=\dim\<H_1,H_3\>=8$ and $\dim\<H_2,H_3\>=7$. Moreover, by our assumption, $\dim\<T_{p'}(H_1),T_{p'}(H_2),T_{p'}(H_3)\>=6+1=7$. This contradicts Lemma~\ref{crucial5}. The claim is proved. 

 Let $y\in X'$ be a point off $\<C,D_1,D_2\>$ and let $z_1$ be a point of $D_1\setminus\{x\}$. Since there is at most one singular line through $y$ meeting $D_1$, we may choose $z_1$ in such a way that there is a conic $E\subseteq X'$ containing $z_1$ and $y$. Now $\<C,D_1,E\>$ coincides with $\PAG^6(\K)$ and hence by Lemma~\ref{crucial5} the conic $E$ intersects $C$, say in the point $u$. By the same token, now observing that $\<E,C,D_2\>$ coincides with $\PAG^6(\K)$, the conic $E$ intersects $D_2$, say in $z_2$. Now the points $u,z_1,z_2$ of $E$ generate $\<E\>$ and hence $E\subseteq\<C,D_1,D_2\>$, contradicting $y\in E$. This contradiction proves the lemma.
\end{proof}



For the rest of this section, we assume that $\dim\<X'\>=6$. Our intention is to show, in a series of lemmas, that this can only happen when $|\K|=2$. 

\begin{lemma}
The set $X'$ does not contain singular $3$-spaces.
\end{lemma}

\begin{proof}
Let, for a contradiction, $U$ be a singular $\ell$-space with $\ell\geq 3$ and $\ell$ maximal.  Since $X'$ contains at least one plane $\pi$ such that $\pi\cap X'$ is a conic, we see that $\ell\leq 4$. If $\ell=4$, then we consider a point $x\in X'$ outside $U$. For any $u\in U$, the line $\<u,x\>$ is non-singular, as repeated use of Lemma~\ref{lemma1} would otherwise lead to a singular subspace of dimension $5$. Pick two distinct points $u,v\in U$. So we have a conic $C\subseteq X$ through $x$ and $u$, and for each point $y$ of $C\setminus\{u\}$, we have a conic  $C_y$ containing $v$ and $y$. Let $y_1,y_2$ be two distinct points of $C\setminus\{u\}$. An arbitrary $5$-space $W$ through $U$ not containing the tangent lines at $v$ to the conics $C_{y_1}$ and $C_{y_2}$, respectively, intersects $C_{y_i}$ in a point $z_i$, $i=1,2$. The line $\<z_1,z_2\>$ intersects $U$ and so is singular, a contradiction.

Next suppose $\ell=3$.  Let $\pi$ be a plane in   $\PAG^6(\K)$ skew to $U$. If $\pi$ contains a conic, then, as before, the projection with center $\pi$ onto $U$ is injective, implying that $X'=U\cup (\pi\cap X')$, a contradiction. Hence every conic intersects $U$. Also, the projection of $X'\setminus U$ with center  $U$ onto $\pi$ is injective, as the line joining two points with same image must meet $U$ and hence is singular, a contradiction as above. If $|\K|=q$ is finite, then there are $q^3+q^2+q+1$ conics through a fixed point of $X'$ outside $U$ and some point of $U$, giving rise to $q^4>q^2+q+1=|\pi|$ points of $X'$ outside $U$, a contradiction. So we may assume that $\K$ is infinite. We claim that no $4$-space through $U$ contains a tangent $T$ to a conic $C$ at a point of $U$ and a point $v$ of $X'\setminus U$. Indeed, let $u\in C\setminus U$. Then the conic $D$ containing  $u$ and $v$ is contained in $\<U,C\>$ (as $D$ contains a point of $U$), contradicting the injectivity of the projection with center $U$ onto $\pi$ ($C$ and $D$ project into the same line). 
This implies that all conics in $X'$ through the same point $u$ of $U$ project onto (``affine'') lines of $\pi$ sharing the same point $p_u$ (``at infinity'') corresponding to the tangents to these conics at $u$. For different $u$, the points $p_u$
are also different as otherwise, by injectivity of the projection, we find two conics through a common point of $X_p\setminus U$ intersecting in all points but the ones in $U$, a contradiction. This now implies that two different conics containing a (possibly different) point of $U$ meet in a unique point of $X$. We choose a plane $\alpha\subseteq U$ and project $X'\setminus \alpha$ from $\alpha$ onto some skew $3$-space $\Sigma$. Let $u_i$, $i=1,2,3$, be three distinct points in $\alpha$. The conics through these points project onto three families of (``affine'') lines such that lines from different families intersect in a unique point.  Considering two families, we see that, since $|\K|>2$, these lie either on a hyperbolic quadric, and the third family cannot exist, or in a plane. In the latter case, we easily see that all points of $X'\setminus U$ are contained in a $5$-space together with $\alpha$, a contradiction considering a conic through some point of $U\setminus \alpha$.
\end{proof}

\begin{lemma}\label{finite}
The field $\K$ is finite.
\end{lemma}

\begin{proof}
Assume $\K$ is infinite. If there is a non-trivial singular subspace in $X'$, then we let $U$ be a singular line; if there is no (non-trivial) singular subspace, then we let $U$ be a quadratic plane. Consider a point $x\in U\cap X'$ and let $\cF$ be the family of all conics in $X'$ containing $x$.  Pick a point $x'\in (U\cap X')\setminus\{x\}$ and let $C_1,\ldots,C_4$ be four distinct conics containing $x'$. By Lemma~\ref{semiaffine}, there are at most four members of $\cF$ that do not intersect all of $C_1,\ldots,C_4$. Hence we can find three members $D_1,D_2,D_3\in\cF$ intersecting all of $C_1,\ldots,C_4$. We now project $X'\setminus U$ with center $U$ onto a complementary space of $U$ in $\<X'\>$. Using Axiom~(MM2), this projection is injective on the set of points not $X'$-collinear with a point of $U$. Hence we see that the projections $C_1',\ldots,C_4'$ of $C_1\setminus\{x\},\ldots,C_4\setminus\{x\}$ are contained in lines which meet the three lines spanned by the projections of $D_1\setminus\{x'\},D_2\setminus\{x'\},D_3\setminus\{x'\}$ in distinct points, and hence these lines are contained in either a plane $\pi$, or a hyperbolic quadric $\cH$ in some $3$-space. Now, every member of $\cF$ intersects at least three of $C_1,\ldots,C_4$. Consequently, the projections of the members of $\cF$ are contained in $\pi$ or in $\cH$, and in the latter case, they are contained in lines belonging to one system of generators, while the projections of the conics through $x'$ are contained in lines of the other system.  Considering a third point $x''$ in $U\cap X'$, we obtain a third set of lines in $\pi$ or in $\cH$. In the latter case, all these lines  intersect infinitely many generators from each system, a contradiction. Hence all points of $X'$ not $X'$-collinear with a point of $U$ are contained, together with $U$, in either a $4$-space (if $U$ is singular), or a $5$-space (if $U$ is a quadratic plane). This is a contradiction as soon as $U$ is not singular, or $U$ is singular and not contained in a singular subspace of  dimension at least $3$ (as otherwise $X'$ does not span a $6$-space). But there are no singular $3$-spaces by the previous lemma. This contradiction completes the proof of the lemma. 
\end{proof}

\begin{lemma}\label{lemma66}
There are no singular planes.
\end{lemma}

\begin{proof}
Suppose there is a singular plane $U$. By Lemma~\ref{finite}, we may assume that $|\K|=q$, a prime power. Consider a point $x\in X'\setminus U$. By joining with points from $U$, we obtain a set of $q^2+q+1$ conics through $x$, and, in view of Lemma~\ref{semiaffine}, at least $q^2+q$ of them must meet an arbitrary conic in $X'$ not containing $x$, contradicting $q^2+q>q+1$. 
\end{proof}

\begin{lemma}
The field $\K$ has only two elements; moreover there are no singular lines in $X'$ and $|X'|=7$ (hence every pair of conics intersects).
\end{lemma}

\begin{proof} By Lemma~\ref{lemma66} and Lemma~\ref{lemma1}, $X'$ cannot contain two intersecting singular lines.
Let $|\K|=q$ (which is allowed by Lemma~\ref{finite}). 
Suppose first that there is some singular line $L$ (and remember that there is no singular plane by the previous lemma). We project $X'\setminus L$ with center $L$ onto a suitable $4$-space $\Theta$.  We consider a conic $C$ containing a point $x$ of $L$. If $x\neq x'\in L$, then Lemma~\ref{semiaffine} assures that all points of $X'$ are obtained by considering all conics through $x'$ and a point of $C$, except possibly for the points on one more conic through $x'$. Hence either there are $q^2+q+1$ points in $X'$ or $(q+1)^2$. In the former case, the projections of the conics through $x$ and $x'$ form the systems of generators (except for one generator of each system) of a hyperbolic quadric in a hyperplane of  $\Theta$, or are contained in a plane, hence $\dim\<X'\>\leq 5$. In the other case, there are exactly $(q+1)^2$ conics meeting $L$. Through a point $z$ of $X'\setminus L$, there are $q+1$ conics meeting $L$, taking account of $q^2+q+1$ points of $X'$; hence there is room for either one more conic, or a singular line $L'$ through $z$. In the latter case, varying $z$ on $L'$, we see that every conic that intersects $L$ also intersects $L'$.  It follows that, if there were at least three singular lines, then all points of $X'$ are contained in the span of these lines, which is $5$-dimensional. Hence we may assume that there are at most two singular lines. But then, there is a conic not intersecting any singular line. We now first look at the case in which there are no singular lines. 

If there are no singular lines, then we consider a point $x$ and a conic $C\not\ni x$. The number of conics through $x$ is either $q+1$ (and then $|X'|=q^2+q+1$ and there are $q^2+q+1$ conics), or $q+2$ (there are no other possibilities by applying Lemma~\ref{semiaffine} to $x$ and $C$), in which case $|X'|=(q+1)^2$ and $|\Xi'|=(q+1)(q+2)$.

Assume now that $q>2$. By the foregoing, we have at most $(q+1)(q+2)$ conics, at least $q^2+q+1$ points and at least one conic intersecting no singular line. As this holds for every residue, Lemma~\ref{remtruuk} implies $q^3+1\leq (q+1)(q+2)$, which is only possible when $q=2$.

If $q=2$ and if we do have singular lines, then we also have a conic not intersecting any singular line. In this case, we have nine points. Note that this is the maximum number of points possible, implying that we have at most twelve conics. But since we have nine points in $X'$, we have nine singular lines in $X_p$ through the point $p'$. The conic in $X'$ not intersecting any singular line corresponds to a symp $H$ of $X_p$ through $p'$. Take a point $x$ on $H$ not $X_p$-collinear with $p'$. Putting $p'=y$ in Lemma~\ref{remtruuk} (and $x$ remains $x$), the conditions are satisfied with $v=9$ and $q=2$, and so the residue through $x$ contains $(v-q-1)q+1=13$ conics, exceeding the maximum number possible, a contradiction.  The same argument works if we have nine points without singular lines, of course. Hence we only have seven points and hence every pair of conics intersects.
\end{proof}

Hence it follows that, if $|\K|>2$, then $(X_p,\Xi_p)$ is a proper Mazzocca-Melone set of split type $3$, which does not exist by Proposition~\ref{d=3}.  But if $|\K|=2$, then either $\dim\<X'\>\leq 5$, or  we have seven points and every pair of conics intersects. Now $(X_p,\Xi_p)$ cannot exist by \cite[Remark 5.5]{JSHVM4}.

 Hence we have shown that proper Mazzocca-Melone sets of split type $5$ cannot exist. Taking account of Corollary~\ref{MMinductive}, this implies:
\begin{prop}\label{cased=579}
Proper Mazzocca-Melone sets of split type $d\in\{5,7,9\}$ do not exist.
\end{prop}

\subsection{Mazzocca-Melone sets of split type at least 10}
Here, we assume that $(X,\Xi)$ is a Mazzocca-Melone set of split type $d\geq 10$. We show that it necessarily is a non-proper one. Note that, this time, we cannot use Lemma~\ref{MMinductive} anymore to apply induction. The idea is now to use Lemma~\ref{badpoints} in a completely different way, namely, to use wrinkles as centers to project the residues in order to get the dimensions small enough so that (MM3) is satisfied. Lemma~\ref{badpoints} assures that (MM2) is preserved, which is not the case for arbitrary projections.

\begin{prop}\label{casengeq5}
Proper Mazzocca-Melone set of split type $d\geq 10$ do not exist.
\end{prop}

\begin{proof}
We use induction to show this, including the cases $d=8,9$, which were already handled. For $d=8$, we moreover assume that the ambient projective space has dimension at most $19$. Then the cases $d=8,9$ do not occur by Propositions~\ref{casen=4} and~\ref{cased=579}. The cases $d=8,9$ are the base of our inductive argument.

Now let $d\geq 10$.
Let $p\in X$ be arbitrary, and let $x\in X$ be collinear to $p$. Our principal aim is to show that $T_p\cap T_x$ has dimension at most $2d-3$.

Consider an arbitrary \EQ\ $H$ through $x$ which does not contain $p$, and which exists by Lemma~\ref{proper}. Suppose first that $T_p$ intersects $\<H\>$ in a subspace $U_p$ of dimension at most $d-3$. Then, as $T_x(H)\cap U_p$ has dimension at most $d-3$, we can find a plane $\pi$ in $T_x(H)$ disjoint from $U_p$. Since the plane $\pi$ has empty intersection with $T_p$, we have $\dim(T_p\cap T_x)\leq 2d-3$.

Now suppose that $U_p$ has dimension at least $d-2$. By Lemma~\ref{subspace}, the set of points of $H$ collinear with $p$ is contained in a maximal singular subspace $M_1$ of $H$. We can now take a second maximal singular subspace $M_2$ of $H$ disjoint from $M_1$. The intersection $M_2\cap U_p$ has dimension at least $\lfloor \frac{d}{2} \rfloor+(d-2)-(d+1)=\lfloor \frac{d}{2} \rfloor-3\geq 2$. Hence we can find a plane $\pi\subseteq M_2\cap U_p$. All points of $\pi$ are wrinkles of $p$.

Let $X'_p$ be the set of points of $X$ collinear with $p$. Lemma~\ref{badpoints} implies that the projection $Y'_p$ of $X'_p$ with center $\pi$ onto a subspace $U$ of $T_p$ complementary to $\pi$ (and of dimension at most $2d-3$) and containing $p$ is injective and induces an isomorphism on the span of any two subspaces $T_p(H_1),T_p(H_2)$. Let $\Xi'_p$ be the family of projections of subspaces $T_p(\xi)$, with $\xi$ ranging through all quadratic spaces containing $p$. Now we consider a hyperplane $V$ of $U$ not containing $p$ and put  $Y_{p,\pi}=U'\cap Y'_p$ and $\Xi_{p,\pi}=\{\xi'\cap U':\xi'\in\Xi'_p\}$. Clearly, $Y_{p.\pi}\cap\xi^*$ is a split quadric and every member of $\Xi_{p,\pi}$ is a $(d-1)$-dimensional space. Moreover, the choice of $\pi$ implies that the pair $(Y_{p,\pi},\Xi_{p,\pi})$ satisfies (MM1) and (MM2).  Now $U'$ has dimension at most $2d-4$, so (MM3) is satisfied trivially. By induction, such a pair does not exist.

Hence our principal aim is proved. Hence we know that, for all $x$ collinear with $p$, the space $T_p\cap T_x$ has dimension at most $2d-3$. This implies that the residue $(X_p,\Xi_p)$ is a proper Mazzocca-Melone set of split type $d-2$ in a projective space of dimension at most $2d-1$. By induction, such a set does not exist.
\end{proof}


\section{Verification of the axioms}\label{verification}

In this section we verify that the Mazzocca-Melone axioms hold in all examples listed in the Main Result. Our approach is almost completely geometric. We only have to know that an embedding in a projective space with given dimension exists. For that, we can refer to the literature (note these embeddings are always established in an algebraic way). The advantage is that we do not have to introduce the rather long algebraic formulae leading to every construction. It also shows that the Mazzocca-Melone axioms are really natural, and we can see geometrically why they have to hold.

So let $(X,\Xi)$ be either a Segre variety $\mathcal{S}_{p,q}(\K)$, $p,q\geq 1$, $p+q\leq 4$, where $\Xi$ is the family of $3$-spaces spanned by the subvarieties isomorphic to $\mathcal{S}_{1,1}(\K)$; or a line Grassmannian variety $\mathcal{G}_{p,1}(\K)$,  $p\in\{4,5\}$, where $\Xi$ is the family of $5$-spaces spanned by the subvarieties isomorphic to $\mathcal{G}_{3,1}(\K)$; or the half-spin variety $\cD_{5,5}(\K)$, where $\Xi$ is the family of $7$-spaces spanned by subvarieties  isomorphic to a half-spin variety $\cD_{4,4}(\K)$; or the variety $\mathcal{E}_{6,1}(\K)$, where $\Xi$ is the family of $9$-spaces spanned by subvarieties isomorphic to the variety $\mathcal{D}_{5,1}(\K)$. Suppose $X$ spans $\PAG^N(\K)$.

We first show that for every member $\xi$ of $\Xi$, the intersection $X\cap\xi$ is the desired subvariety (which is always isomorphic to some hyperbolic quadric).

This is clear for the Segre varieties and the line Grassmannian varieties since their definition implies immediately that every Segre or line Grassmannian, respectively,  subvariety induced by a suitable subspace of the underlying projective space is also induced by a suitable subspace of the ambient projective space. Since the half-spin variety $\mathcal{D}_{5,5}(\K)$ appears as residue in the variety $\mathcal{E}_{6,1}(\K)$,  it suffices to show the result for the latter.

We know that $N=26$ in this case and that the residue in a point $p\in X$ is isomorphic to a half-spin variety $\mathcal{D}_{5,5}(\K)$, which lives in a space of dimension at most 15 (note that we ignore Remark~\ref{secant} here, which says that this dimension is precisely 15); hence the space $T_p$ generated by all singular lines through $p$ has dimension at most 16 (which already shows (MM3)). Now consider any symp $H$ which is opposite $p$ in the corresponding building of type $\mathsf{E}_6$. Then no point of $H$ is collinear with $p$ in $X$. We claim that $T_p$ and $H$ generate the whole space $\PAG^{26}(\K)$. Indeed, let $x\in X$ be an arbitrary point, which does not belong to $H$ and which is not collinear with $p$. By 3.7 of \cite{TitsE6}, there is a unique symp $H'$ containing $p$ and $x$. Then, by the same reference, $H\cap H'$ is some point $x'$. Then $x'$ and the points of $H'$ collinear with $p$ generate $\<H'\>$, hence $x$ belongs to that space and the claim is proved. It follows that all points lie in the space generated by $T_p$ and $H$. Since the former has dimension at most 16 and the latter dimension $9$, and since the whole space has dimension 26, we see that $T_p$ and $\<H\>$ are disjoint. In particular, $\<H\>$ does not contain $p$. Since $p$ and $H$ are essentially arbitrary, this shows that the space generated by any symp does not contain any point ``opposite'' that symp. Now suppose $\<H\>$ contains a point $x\in X\setminus H$ not opposite $H$. Then, by 3.5.4 and 3.9 of \cite{TitsE6}, there is a $4$-space $U$ in $H$ contained in a singular $5$-space $U'$ together with $x$. But Lemma~\ref{quadric} implies that $U'$ contains points of $H$ outside $U$, a contradiction. Our assertion is proved.

\begin{rem}\label{mm2}\rm In the previous argument, the two symps $H$ and $H'$ can be regarded as an arbitrary pair of symps meeting in one point. Since $T_p(H')$ belongs to $T_p$ and is hence disjoint from $\<H\>$, we see that $\<H\>\cap\<H'\>=\{x'\}$, proving (MM2) in this case.
 \end{rem}

We now verify the axioms for $(X,\Xi)$.

Axiom~(MM1) follows in each case directly from the (geometric) definition of the variety.

It is readily seen that the validity of (MM2) is inherited by the residues. Hence we only need to check (MM2) for the three varieties $\mathcal{S}_{2,2}(\K)$, $\mathcal{G}_{5,1}(\K)$ and $\mathsf{E}_{6,1}(\K)$.  Let us concentrate on the latter; for the former the proofs are similar.

Let $H$ and $H'$ be two symps. There are two possibilities. If $H\cap H'$ is a singleton, then (MM2) follows from Remark~\ref{mm2}. If $H\cap H'$ is a 4-space, then (MM2) follows from Lemma~\ref{quadric} and the fact that $\<H\>$ does not contain points of $H'\setminus H$.

Axiom (MM3) follows immediately from the dimensions of the universal embeddings of the residues.

This completes the proof of the Main Result.

\begin{rem}\rm
There is a variation of Axiom (MM3) involving the dimension of the space generated by the tangent spaces at a point $x$ to the symps intersecting a given singular line none of whose points is collinear to $x$. In that case, one characterizes the same varieties as in the Main Result with $d\geq 2$, but additionally all Segre varieties $\cS_{p,q}(\K)$ for arbitrary $p,q\geq 1$, and all line Grassmannian varieties $\cG_{p,1}(\K)$, for $p\geq 3$. The proofs rely on the Main Result of the present paper, and will appear elsewhere. This result shows that our approach can also include the higher dimensional FTMS (the North-West $3\times 3$ square for higher dimensions).
\end{rem}

{\bf Acknowledgement} The authors would like to thank the Mathematisches Forschungsinstitut Oberwolfach for providing them with the best possible research environment imaginable during a stay as Oberwolfach Leibniz Fellow (JS) and Visiting Senior Researcher (HVM). During this stay substantial progress on this article was made.



\begin{thebibliography}{99}

\bibitem{Asch1} 
M.~Aschbacher,
The 27-dimensional module for E6, I.
{\em Invent. Math.} \textbf{89} (1987), 159--195.

\bibitem{BrCo}
{A.E. Brouwer \& A.M. Cohen} Some remarks on Tits geometries.
With an appendix by J. Tits.
\emph{Nederl. Akad. Wetensch. Indag. Math.} {\bf 45} (1983), no. 4, 393�-402.

\bibitem{BC}
{F. Buekenhout \& P. Cameron}, {\em Projective and Affine Geometry
over Division Rings.} \textbf{In:} Buekenhout, F. (ed.), Handbook of Incidence
Geometry, Amsterdam, Elsevier (1995), 27--62.

\bibitem{Bourbaki} N. Bourbaki, \emph{Groupes et Alg\`ebres de Lie}, Chapters 4,5 and 6. \emph{Actu. Sci. Ind.} \textbf{1337}, Hermann, Paris, 1968.

\bibitem{Chaput} P.-E.~Chaput, Severi varieties. \emph{Math. Z.}~\textbf{240} (2002), 451--459.

\bibitem{Chevalley}
{C. Chevalley}, \emph{The algebraic theory of spinors.} Columbia University Press, New York, 1954.


\bibitem{Coo-Shu} B. Cooperstein \& E.~E.~Shult, Frames and bases of Lie incidence geometries. \emph{J. Geom.} \textbf{60} (1997), 17--46.

\bibitem{BDB-HVM} B. De Bruyn \& H. Van Maldeghem, Dual polar spaces of rank 3 defined over quadratic alternative division algebras. \emph{J. Reine Angew. Math.}, to appear. 

\bibitem{GH} P. Griffiths \& J. Harris, Algebraic geometry and local differential geometry, \emph{Ann. Sci. ENS} \textbf{12} (1979), 335--432.

\bibitem{Hav} H.~Havlicek, Zur Theorie linearer Abbildungen I, II. \emph{J. Geom.} \textbf{16} (1981),
152--180.

\bibitem{Hir-Tha:91} J.~W.~P.~Hirschfeld \& J.~A.Thas, \emph{General Galois Geometries}. Oxford Mathematical Monographs, Oxford Science Publications, The Clarendon Press, Oxford University Press, New York, 1991.



\bibitem{IM}
{A. Iliev \& D.Markushevich}, Elliptic curves and rank-2 vector bundles on the prime Fano threefold of genus 7.
\emph{Adv. Geom.} (2004) {\bf 4}, 287--318.

\bibitem{Ion-Rus:08}  P. Ionescu \& F. Russo, Varieties with quadratic entry locus $II$.
\emph{Compositio Math.} \textbf{144} (2008),  949--962.

\bibitem{Ion-Rus:10} P. Ionescu \& F. Russo,  Conic-connected manifolds.  \emph{J. Reine Angew. Math.} \textbf{644} (2010), 145--158.

\bibitem{KS} 
{A. Kasikova \& E.E. Shult}
Absolute embeddings of point-line geometries. \emph{J. Algebra}, {\bf 238}, (2001), 265--291.

\bibitem{KSVM}
O. Krauss, J. Schillewaert \& H. Van Maldeghem, Veronesean representations of projective planes over quadratic alternative division algebras. To appear in \emph{Michigan Math. J.}


\bibitem{Laz-AVV:84} R. Lazarsfeld \& A. Van de Ven, \emph{Topics in the Geometry of Projective Space. Recent work of F. L. Zak.} Birkh\"auser Verlag, Basel, Boston, Stuttgart, 1984.

\bibitem{Manivel}
L.~Manivel, On spinor varieties and their secants. {\em SIGMA Symmetry Integrability Geom. Methods Appl.} {\bf 5} (2009), Paper 078, 22 pp.




\bibitem{Maz-Mel:84} F.~Mazzocca \& N.~Melone, Caps and Veronese varieties in projective Galois spaces. \emph{Discrete Math.} \textbf{48} (1984), 243--252.

\bibitem{Mukai} S. Mukai, Curves and symmetric spaces, I.
{\em Amer. J. Math.} {\bf 117} (1995), 1627--1644.

\bibitem{Nash} O.~Nash, K-theory, LQEL manifolds and Severi varieties. \emph{Geom. Top.} \textbf{18} (2014) 1245--1260.
 
\bibitem{Rus:09}  F. Russo,  Varieties with quadratic entry locus $I$.  \emph{Math.
Ann.} \textbf{344} (2009), 597--617.

\bibitem{JSHVM1}
{J.~Schillewaert \& H.~Van Maldeghem,}
Quadric Veronesean caps. \emph{Bull. Belgian Math. Soc. Simon Stevin} \textbf{20} (2013), 19--25.

\bibitem{JSHVM2}
{J.~Schillewaert \& H.~Van Maldeghem,}
Hermitian Veronesean caps. {\em Springer Proceedings in Mathematics}, {\bf 12} (2012), 175--191.


\bibitem{JSHVM3}
{J.~Schillewaert \& H.~Van Maldeghem,}
Projective planes over quadratic 2-dimensional algebras.  \emph{Adv. Math.} \textbf{262} (2014), 784Ð822.

\bibitem{JSHVM4}
{J.~Schillewaert \& H.~Van Maldeghem,} A combinatorial characterization of the Lagrangian Grassmannian LG$(3,6)$. To appear in \emph{Glasgow Math. J.}

\bibitem{Severi}
{F. Severi}, Intorno ai punti doppi impropri di una superficie generale dello spazio a quattro
dimensioni, e a suoi punti tripli apparenti, \emph{Rend. Circ. Mat. Palermo} \textbf{15} (1901), 33--51.


\bibitem{Shu:12} E.~E.~Shult, \emph{Points and Lines, Characterizing the Classical Geometries}. Universitext, Springer-Verlag, Berlin Heidelberg, 2011.

\bibitem{Tha-Mal:04b} J.~A.~Thas \& H. Van Maldeghem, Characterizations of the
finite quadric Veroneseans $\mathcal{V}_n^{2^n}$. \emph{Quart. J.
Math.}~\textbf{55} (2004), 99--113.


\bibitem{Tits1}
{J.~Tits,} Groupes semi-simples complexes et g\'eom\'etrie projective. \emph{S\'eminaire Bourbaki} \textbf{7} (1954/1955)

\bibitem{TitsThesis} J. Tits, Sur certaines classes d'espaces homog\`enes de groupes de Lie,
{\it Acad. Roy. Belg. Cl. Sci. M\'em. Collect. 8\,$^o$} (2)
 {\bf 29}\,(3) (1955), 268 p.
 
\bibitem{TitsE6} J. Tits, Sur la g\'eometrie des $R$-espaces. \emph{J. Math. Pure Appl.} (9) \textbf{36} (1957), 17--38.

\bibitem{Tits2} {J.~Tits},
Alg\`ebres alternatives, alg\`ebres de Jordan et alg\`ebres de Lie exceptionnelles. \emph{Indag. Math.} \textbf{28} (1966) 223--237.

\bibitem{Tits83} J.~Tits, A local approach to buildings. in {\it The Geometric vein. The Coxeter
Festschrift}, Coxeter symposium, University of Toronto, 21--25 May 1979,
Springer-Verlag, New York 1982, 519--547.


\bibitem{Wells} {A.~L.~Wells Jr},
Universal projective embeddings of the Grassmannian, half spinor, and dual orthogonal geometries. \emph{Quart. J. Math. Oxford} \textbf{34} (1983), 375--386.

\bibitem{Zak}
{F. Zak,}
Tangents and secants of algebraic varieties. {\em Translation of mathematical monographs}, AMS, 1983.

\bibitem{Zan} C. Zanella, Universal properties of the Corrado
Segre embedding. \emph{Bull. Belg. Math. Soc. Simon Stevin} \textbf{3} (1996), 65--79.

\end{thebibliography}
\end{document}